\documentclass[11pt,a4paper]{article}
\usepackage[T1]{fontenc}
\usepackage[utf8]{inputenc}
\usepackage{lmodern}
\usepackage{amsmath,amssymb,amsthm,mathtools,mathrsfs}
\usepackage{booktabs,tabularx,array,enumitem}
\usepackage{microtype,xurl}
\usepackage[margin=0.88in]{geometry}
\usepackage{xcolor}
\usepackage{cite}
\usepackage{caption}
\usepackage{needspace}
\definecolor{linkblue}{RGB}{26,52,108}
\usepackage[colorlinks=true,linkcolor=linkblue,citecolor=linkblue,urlcolor=linkblue]{hyperref}
\usepackage[nameinlink,noabbrev]{cleveref}
\hypersetup{pdftitle={Matching Randomized First-Order Bounds for Higher-Order Smooth Nonconvex Optimization},pdfauthor={},pdfsubject={Oracle complexity and randomized lower bounds},pdfkeywords={oracle complexity, randomized lower bounds, upper bounds, nonconvex optimization, higher-order smoothness}}
\allowdisplaybreaks[1]
\newtheorem{theorem}{Theorem}[section]
\newtheorem{lemma}[theorem]{Lemma}
\newtheorem{proposition}[theorem]{Proposition}
\newtheorem{corollary}[theorem]{Corollary}
\theoremstyle{definition}

\theoremstyle{remark}
\newtheorem{remark}[theorem]{Remark}
\newcommand{\ResearchAgentSystem}{Colombo, our laboratory's internal auto-research system}
\newcommand{\R}{\mathbb R}
\newcommand{\E}{\mathbb E}
\newcommand{\eps}{\varepsilon}
\newcommand{\norm}[1]{\left\|#1\right\|}
\newcommand{\inner}[2]{\left\langle #1,#2\right\rangle}
\newcommand{\op}{\mathrm{op}}
\newcommand{\Lip}{\operatorname{Lip}}
\newcommand{\Span}{\operatorname{span}}
\newcommand{\spanop}{\operatorname{span}}
\newcommand{\range}{\operatorname{range}}
\newcommand{\cyc}{\operatorname{cyc}}
\newcommand{\diag}{\operatorname{diag}}
\newcommand{\Sym}{\operatorname{Sym}}
\newcommand{\KL}{D_{\mathrm{KL}}}
\newcommand{\TV}{d_{\mathrm{TV}}}
\newcommand{\Ptrue}{\mathbb P}
\newcommand{\Pref}{\mathbb Q}
\newcommand{\vis}{\mathrm{vis}}
\newcommand{\hid}{\mathrm{hid}}
\newcommand{\Qrand}{Q_p^{\infty,\mathrm{rand}}}
\newcommand{\Qdet}{Q_p^{\infty,\mathrm{det}}}

\newcolumntype{Y}{>{\raggedright\arraybackslash}X}
\title{\textbf{Matching Randomized First-Order Bounds for Higher-Order Smooth Nonconvex Optimization}}
\newcommand{\authornames}{Haihan Zhang, Wendao Wu, Fangqing Zhang, Chenheng Zhang, Chunyuan Zheng, Cong Fang, Haoxuan Li, Zhouchen Lin}
\newcommand{\affiliation}{Peking University}
\newcommand{\email}{colombo.2026@outlook.com}
\newcommand{\mail}[1]{\href{mailto:#1}{\texttt{#1}}}
\newif\ifanonymous
\anonymousfalse
\ifanonymous
\author{Anonymous Authors}
\hypersetup{pdfauthor={Anonymous Authors}}
\else
\author{Haihan Zhang$^{1,*}$\quad Wendao Wu$^{1,*}$\quad Fangqing Zhang$^{1,*}$\\[0.20em] Chenheng Zhang$^{1,*}$ \quad
Chunyuan Zheng$^1$\\[0.20em]
Cong Fang$^{1,\dagger}$\quad Haoxuan Li$^{1,\dagger}$\quad Zhouchen Lin$^{1,\dagger}$\\[0.60em]
\small $^1$\affiliation\\[0.35em]
\footnotesize\mail{\email}\\[0.35em]
\small $^*$Equal contribution.\quad $^\dagger$Corresponding authors.}
\hypersetup{pdfauthor={\authornames}}
\fi
\date{}
\begin{document}
\maketitle
\begin{abstract}
Does randomization improve the worst-case first-order complexity of nonconvex
optimization when the objective is very smooth? We prove that, for every fixed
finite integer $p\ge3$, finding an $\eps$-stationary point requires
$\Omega_p(\eps^{-5/3})$ exact value-and-gradient queries, even for $C^\infty$
objectives with initial gap at most one and simultaneous global Lipschitz
budgets of one for derivatives of orders $1,\ldots,p$. Together with the known
deterministic upper bound, this determines the accuracy exponent up to a
single logarithmic factor. The lower bound holds for unrestricted adaptive
randomized algorithms, unbounded query locations, and outputs that need not
have been queried. One algorithm-independent hard distribution suffices in
dimension $O_p(\eps^{-47/3})$. The construction combines a block-chain potential
with an exact quartic core, repeated quadratic spectra, and two static Gaussian
tensor perturbations. A scheduled Haar argument controls the entire adaptive
query span. An innovation-only Gaussian comparison then bounds exact transcript
information without assuming finite precision or a minimum query separation.
This establishes the $5/3$ randomized accuracy barrier under every prescribed
finite smoothness order, extending the previously matched third-order case.
\end{abstract}

{\small
\begin{samepage}
\noindent\textbf{AI Usage.}
Nearly the entire research pipeline for this paper was carried out by
\ResearchAgentSystem{}, powered by GPT-5.6 Sol. The system also conducted a
Lean-backed article audit of the resulting manuscript. The authors subsequently reviewed and approved the
mathematical claims, presentation, and formal artifacts, and take
responsibility for the final manuscript. The complete Lean audit report and the system's technical report will be made public at a later date.
\par
\end{samepage}
}

{\small
\begin{samepage}
\paragraph{Publication Intent and Recognition of Independent Work.}

This preprint is released as a research output of our study of
automated scientific research, for public scrutiny and discussion.
As we do not intend to submit this manuscript to a journal or
conference, it has not undergone the full editorial refinement
and internal review we would normally require for submission.
The authors remain responsible for its contents and for correcting
any errors or inaccurate attribution.

We recognize that releasing these results may affect researchers
who have independently been pursuing the same questions, including
work that is not yet public.
If this applies to your work, we invite you to contact us at
\email.
A brief description of the overlap, together with any dated
materials you are comfortable sharing, would help us understand
its scope and chronology.
Such materials might include draft excerpts, presentation slides,
or revision records; a complete manuscript or unpublished proof
is not required to begin the conversation.

We will review these cases carefully and update citations,
acknowledgments, and accounts of independent or concurrent work
as appropriate, including revising priority claims where warranted.
Nonpublic materials will be treated confidentially and used only
to assess the overlap unless their authors authorize otherwise.
We will obtain explicit permission from the authors before
publicly discussing such materials or providing them to an
automated research system.
\par
\end{samepage}
}

\section{Introduction}\label{sec:introduction}
Consider minimizing a smooth nonconvex function when an algorithm may query
exact function values and gradients, but no higher-order derivatives. How many
queries are necessary to find \(x\) with
\(\|\nabla f(x)\|\le \eps\)? Under a Lipschitz-gradient assumption, the
familiar \(\eps^{-2}\) accuracy dependence is worst-case optimal, including
for randomized algorithms~\cite{lower1}. Higher-order smoothness changes this
picture even though it does not enrich the oracle: Lipschitz third derivatives
allow a deterministic first-order method to achieve
\(O(\eps^{-5/3}\log(1/\eps))\) queries~\cite{carmon}.

This raises the central question of the paper. If the objective obeys a fixed
finite hierarchy of higher-order smoothness bounds, can private randomization
improve the \(5/3\) exponent while the algorithm still observes only values and
gradients? The answer identifies whether the accelerated rate is an artifact
of current deterministic methods or an information barrier for first-order
optimization. It also separates two resources that are often conflated:
additional regularity of the objective and additional information returned by
the oracle.

The known bounds do not resolve this question. Zhou~\cite{zhou} obtains the
\(5/3\) deterministic exponent through a block-chain construction, but a
deterministic resisting oracle may choose its orientation after seeing the
query sequence. It therefore does not produce a function fixed before a
randomized algorithm draws its private seed. Zhou, Sugiyama, and
Suzuki~\cite{trapping} prove a randomized exponent \(3/2+1/(2p)\) when the
gradient and the \(p\)th derivative have fixed Lipschitz budgets. This matches
\(5/3\) for \(p=3\), but is smaller for \(p\ge4\); their all-\(p\) randomized
\(5/3\) result instead removes the fixed gradient-Lipschitz budget. The open
regime is therefore the normalized class with all budgets
\(L_1,\ldots,L_p\) fixed and \(p\ge3\) finite.

\paragraph{Main result.}
For the normalized class defined in Section~\ref{sec:main}, we prove
\begin{equation}\label{eq:intro-rate}
 c_p\eps^{-5/3}\le\Qrand(\eps)
 \le\Qdet(\eps)\le C_p\eps^{-5/3}\log(2/\eps).
\end{equation}
Thus randomization does not improve the dimension-uniform first-order accuracy
exponent under any prescribed finite order of smoothness. The upper and lower
bounds match up to one logarithmic factor. The supremum over finite dimensions
is part of the complexity definition: the hard dimension may grow with the
accuracy requirement, and the result is not a fixed-dimensional lower bound.
More strongly, the lower bound is witnessed by one distribution, chosen
independently of the algorithm and its private seed, in dimension
\(O_p(\eps^{-47/3})\). Our contribution is this randomized lower bound. The
upper bound in~\eqref{eq:intro-rate} follows from the existing first-order
method of Carmon et al.~\cite{carmon}; a one-dimensional quadratic lift makes
its guarantee uniform even when the objective's actual initial gap is
arbitrarily small.

\paragraph{Why exact information is the main difficulty.}
A small hidden component in a gradient is not the same as no information.
An exact oracle allows arbitrarily small rescalings and differences between
nearly dependent queries. Moreover, a perturbation sampled once as part of the
objective cannot be treated as fresh noise at each call: repeated queries must
return the same answer. We address both issues by controlling the whole query
span and by comparing only genuinely new Gaussian observation components.

\paragraph{Construction and proof ideas.}
The lower bound combines four ingredients, each addressing a distinct obstacle.
\begin{enumerate}[leftmargin=*,label=(\roman*)]
\item We retain the entry-to-suffix block-chain architecture of~\cite{zhou},
but replace its scalar potential by a smooth function with an exact quartic
core. The core yields a cubic-plus-quartic tensor signal, while its controlled
tails permit one simultaneous calibration of every prescribed derivative
budget. This is what keeps the \(5/3\) exponent unchanged as \(p\) increases.

\item We repeat the quadratic spectrum and apply one Haar rotation. High
spectral multiplicity then limits the alignment of every undisclosed readout
with the entire adaptive query span, including directions obtained by
subtracting nearly dependent queries and the final unqueried output.

\item We conceal the cubic and quartic signals with two static Gaussian
tensors. An innovation-only Gaussian comparison charges KL divergence only to
new tensor-observation directions. Because the perturbations are sampled once,
the comparison preserves exact replay and remains valid for singular
measurements and arbitrary query scales. This converts geometric hiding into a
fixed, algorithm-independent hard distribution rather than an adaptive oracle.

\item A global radial transformation confines every internal query while
preserving a stationarity obstruction both inside and outside the truncation
radius. It enforces the global smoothness and gap constraints without bounding
the algorithm's query locations or requiring its output to have been queried.
\end{enumerate}
Together, these ingredients close the randomized gap under all normalized
budgets \(L_1,\ldots,L_p\). The appendices give the analytic calibration, Haar
geometry, exact transcript comparison, and full quantifier conversion.

\section{Problem formulation and main result}\label{sec:main}
Fix an integer $p\ge3$, independently of $\eps$. All vector norms are Euclidean.
For $x,v\in\R^d$ with $\|v\|=1$, write $f_{x,v}(t)=f(x+tv)$.
Let $\mathcal F_{1:p}^{(d)}(\Delta,L_1,\ldots,L_p)$ be the set of
$f\in C^\infty(\R^d)$ such that $f(0)-\inf f\le\Delta$ and
\begin{equation}\label{eq:function-class}
 |f_{x,v}^{(q)}(t)-f_{x,v}^{(q)}(s)|\le L_q|t-s|
 \quad(q\in[p],\ x\in\R^d,\ \|v\|=1,\ s,t\in\R).
\end{equation}
This is the directional smoothness convention of~\cite{carmon,zhou}.
Define $\mathcal H_{p,d}=\mathcal F_{1:p}^{(d)}(1,1,\ldots,1)$.
Third-order smoothness means a Lipschitz third derivative, not merely a
bounded third derivative. Only the first $p$ derivative budgets are prescribed;
we do not impose one common bound at all derivative orders.

\paragraph{Oracle and algorithms.}
A query at any finite $x\in\R^d$ returns
$\mathcal O_f(x)=(f(x),\nabla f(x))$ exactly. Computation between calls and
query norms are unrestricted. Policies and terminal outputs are measurable,
and the output may be unqueried. Each objective is fixed independently of the
algorithm's private seed. Let $Q_p^{\mathrm{rand}}(d,\eps)$ be the least integer
query cap $T$ for which one algorithm makes at most $T$ calls almost surely
on each fixed $f\in\mathcal H_{p,d}$ and satisfies
\begin{equation}\label{eq:success-definition}
 \Pr_\omega\{\|\nabla f(x_{\mathrm{out}})\|\le\eps\}\ge2/3
 \qquad\text{for every fixed }f\in\mathcal H_{p,d}.
\end{equation}
The deterministic version requires guaranteed success. Set
$\Qrand(\eps)=\sup_{d\ge1}Q_p^{\mathrm{rand}}(d,\eps)$, and define
$\Qdet$ analogously. These are worst-case query caps, not expected costs.
Throughout, private-seed, advice, side-information, and transcript spaces
are standard Borel. Randomized policies are measurable kernels (equivalently,
after adjoining independent $[0,1]$ seeds, measurable functions). Thus the
regular conditional laws, disintegrations, and Fubini--Tonelli steps used
below exist in their standard measurable versions.

\begin{theorem}[Matching randomized accuracy exponent]\label{thm:matching}
For every fixed finite $p\ge3$, there are constants
$c_p,C_p,\eps_p>0$ such that, for $0<\eps\le\eps_p$,
\begin{equation}\label{eq:main-rate}
 c_p\eps^{-5/3}\le\Qrand(\eps)
 \le\Qdet(\eps)\le C_p\eps^{-5/3}\log(2/\eps).
\end{equation}
More precisely, there is an integer $d_\eps\le C_p\eps^{-47/3}$ and a
single distribution $\mathcal D_{p,\eps}$ supported on
$\mathcal H_{p,d_\eps}$ such that every randomized policy with a pathwise cap
of $\lfloor c_p\eps^{-5/3}\rfloor$ calls has average success probability
at most $1/3$ under $\mathcal D_{p,\eps}$. The distribution is independent
of the policy and of its private seed.
\end{theorem}

For $p\ge4$, the exponent improvement over $3/2+1/(2p)$ is
$(p-3)/(6p)$; for $p=3$ we recover an already matched exponent by a different
hard distribution. Constants and the small-accuracy threshold may depend on
$p$. The theorem does not remove the logarithm, optimize the hard dimension,
or yield a bound uniform as $p$ grows with $1/\eps$.

\paragraph{Notation.}
$D^qf$ is the symmetric $q$-linear Fr\'echet derivative. We use the full
multilinear operator norm for hard-instance derivative estimates, the
Frobenius norm for tensor observation spaces, and $P_V$ for orthogonal
projection onto a subspace $V$. $A^\dagger$ denotes the Moore--Penrose
pseudoinverse. $\Sym$ averages over tensor permutations. A standard Gaussian
in a finite-dimensional tensor space is isotropic for its Frobenius inner
product. Subscripts $p$ on asymptotic notation permit constants depending
only on the prescribed order. Logarithms are natural.

\section{Related work}\label{sec:related}
\paragraph{Accelerated first-order upper bounds.}
Carmon et al.~\cite{carmon} give deterministic exact first-order bounds of
$O(\eps^{-7/4}\log(1/\eps))$ under Lipschitz Hessians and
$O(\eps^{-5/3}\log(1/\eps))$ under Lipschitz third derivatives.
Li and Lin~\cite{lilin23} remove the logarithm in the Hessian-Lipschitz
case. Xiong, Jerad, and Cartis~\cite{xiong} obtain the third-order gradient
complexity with a parameter-free method that also uses values for backtracking.
We use the value-and-gradient bound of~\cite{carmon}, rather than identifying
a gradient-evaluation count with the total cost of a different oracle model.
The related Hessian-free accelerated method of Carmon et
al.~\cite{accelerated} additionally guarantees approximate second-order
stationarity.

\begin{table}[t]
\caption{Accuracy dependence in the high-dimensional exact first-order model,
with initial gap normalized to one. Lower bounds allow a sufficiently large
finite dimension. \textsuperscript{a}The intermediate budgets follow from
Lemma~6.12 of~\cite{trapping} by the calibration explained below.
The $p=2$ row is a separate smoothness regime. References~\cite{zhou,zhang,trapping} are 2026 preprints.}
\label{tab:frontier}
\centering
\small
\setlength{\tabcolsep}{5pt}
\renewcommand{\arraystretch}{1.18}
\begin{tabularx}{\textwidth}{@{}>{\raggedright\arraybackslash}p{0.28\textwidth}>{\raggedright\arraybackslash}p{0.20\textwidth}>{\raggedright\arraybackslash}p{0.15\textwidth}Y@{}}
\toprule
\textbf{Result} & \textbf{Normalized smoothness} & \textbf{Algorithms} & \textbf{Query bound} \\
\midrule
Carmon et al.~\cite{carmon} & $L_1=L_3=1$ & Deterministic &
$O(\eps^{-5/3}\log(1/\eps))$ (upper bound) \\[7pt]
Carmon et al.~\cite{lower2} & $L_1=\cdots=L_p=1$, fixed $p\ge3$ & Deterministic &
$\Omega_p(\eps^{-8/5})$ (lower bound) \\[7pt]
Zhou~\cite{zhou} & $L_1=\cdots=L_p=1$, fixed $p\ge3$ & Deterministic &
$\Omega_p(\eps^{-5/3})$ (lower bound) \\[7pt]
Zhang et al.~\cite{zhang} & $L_1=L_2=1$ & Randomized &
$\Omega(\eps^{-7/4})$ (lower bound) \\[7pt]
Zhou et al.~\cite{trapping} & $L_1=\cdots=L_p=1$, fixed $p\ge3$\textsuperscript{a} & Randomized &
$\Omega_p(\eps^{-3/2-1/(2p)})$ (lower bound) \\[7pt]
\textbf{Theorem~\ref{thm:matching}} & $L_1=\cdots=L_p=1$, fixed $p\ge3$ & Randomized &
$\Omega_p(\eps^{-5/3})$ (lower bound) \\[7pt]
\bottomrule
\end{tabularx}
\end{table}

\paragraph{Deterministic lower bounds and the block-chain starting point.}
Carmon et al.~\cite{lower2} establish first-order lower bounds of
$\Omega(\eps^{-12/7})$ for Lipschitz Hessians and
$\Omega_p(\eps^{-8/5})$ under arbitrarily many fixed smoothness constraints.
Zhou~\cite{zhou}, Theorem~5.2, obtains the sharp deterministic exponents
$7/4$ and $5/3$. Our analytic starting point is the entry-to-suffix
block-chain architecture of that work; we do not claim this architecture as
new. The scalar potential is modified, and the randomized information argument
is separate. A deterministic lower bound, including one obtained by reducing
to zero-respecting algorithms, does not alone provide the fixed-distribution
statement in Theorem~\ref{thm:matching}.

\paragraph{Recent randomized lower bounds.}
Zhang et al.~\cite{zhang} treat the exact value-and-gradient Hessian-Lipschitz
problem using large eigenspaces and sequential disclosure. The higher-order
bound of Zhou, Sugiyama, and Suzuki~\cite{trapping}, Theorem~6.1 and
Corollary~6.3, has exponent $3/2+1/(2p)$; Remark~6.2 allows exact values and
unqueried outputs. Its Lemma~6.12 bounds intermediate derivatives by
$C_q\mu^q/\eps^{q-1}$. Choosing
$\mu=a_p\eps^{(p-1)/p}$ gives $C_qa_p^q\eps^{1-q/p}\le1$ simultaneously for
$2\le q\le p$, with $a_p$ sufficiently small. Thus its comparison row applies
to our entire normalized class, not just to the endpoint budgets. The
$5/3$ bound obtained there without a fixed $L_1$ constraint concerns a larger
class. Our improvement for $p\ge4$ retains $L_1$ and every intermediate
budget. High-multiplicity spectra and delayed disclosure have antecedents in
these works; the additional mechanism here is a polynomial signal hidden by
static tensors and an exact, innovation-only transcript comparison.

\paragraph{Higher-order, mixed, and stochastic oracles.}
With derivatives through order $p$ supplied by the oracle, regularized Taylor
methods achieve $O(\eps^{-(p+1)/p})$ evaluations~\cite{birgin}; matching
randomized lower bounds are developed by Carmon et al.~\cite{lower1}.
These are not pure first-order query bounds. Adil et al.~\cite{adil} study
tradeoffs between gradient and Hessian queries, including dimension-dependent
first-order consequences. Stochastic first-order lower bounds~\cite{arjevani}
concern random oracle errors. In our construction the Gaussian perturbation is
sampled once and then remains part of a deterministic function; repeated
queries are exact repetitions, not new noisy samples.

\paragraph{Dimension and adaptive linear observations.}
The dependence on dimension cannot be suppressed when comparing different
settings: already in one dimension, the interaction of randomization and
value access changes stationarity complexity~\cite{chewi}. Trapping methods
in~\cite{trapping} likewise provide dimension-dependent improvements.
Our geometric argument is related to exact adaptive matrix-vector query
lower bounds~\cite{simchowitz,sun}. Its specific requirement is stronger than
a bound on individual query correlations: it bounds the projection of every
undisclosed readout onto the entire adaptive query span while permitting
scheduled disclosure and auxiliary Gaussian reference replies.

\section{A globally smooth hard distribution}\label{sec:construction}
We describe the construction before turning to its information content.
Appendices~\ref{app:analytic} and~\ref{app:global} give all scalar estimates,
constants, and global derivative bounds.

\subsection{An exact quartic core preserves every prescribed smoothness budget}
Choose an even $C^\infty$ cutoff $\chi$ equal to one on $[-1,1]$ and zero
outside $[-2,2]$. Put $W(t)=\chi(t)+(1-\chi(t))t^{-2}$, with $W(0)=1$, and
\begin{equation}\label{eq:overview-primitive}
 \phi_r(u)=120\int_1^u v^2(v-1)W(v/r)\,dv,\qquad r\ge1.
\end{equation}
Then $\phi_r\ge0$, $\phi_r(1)=0$, and
\begin{equation}\label{eq:overview-core}
 \phi_r(u)=30u^4-40u^3+10\ (|u|\le r),\qquad
 \|\phi_r^{(q+1)}\|_\infty\le C_qr^{3-q}.
\end{equation}
The potential is polynomial only on its core; its derivative becomes linear
in the tails, making all the required global derivative bounds possible.

Let $H=\gamma I+L\mathsf P_M$, where $\mathsf P_M$ is the path Laplacian and
$L>0$ is a small constant. In each block choose a vector $a$ of norm $A$,
supported on the suffix $j\ge j_\star:=\lceil M/2\rceil$, and let $e=e_1$.
For $z=(z_1,\ldots,z_N)$ write $u_s=a^\top z_s$ and $u_0=1$. Set
\begin{equation}\label{eq:overview-block}
 F(z)=\sum_{s=1}^N\left[
 \tfrac12z_s^\top Hz_s-\lambda u_{s-1}e^\top z_s
 +\tfrac\kappa2u_s^2+\tau\phi_r(u_s)\right],
\end{equation}
where $\lambda=(a^\top H^{-1}e)^{-1}$,
$\tau=(a^\top H^{-1}a)^{-1}$, and
$\kappa=(e^\top H^{-1}e)/(a^\top H^{-1}e)^2$.
The pure quadratic matrix $Q_0$ is positive semidefinite, and the linear term
is $c_0=(-\lambda e,0,\ldots,0)$.

For any fixed small derivative budget $b>0$, the calibration in
Proposition~\ref{ana:prop:calibration} gives
\begin{equation}\label{eq:overview-scales}
 \begin{gathered}
 \gamma\asymp_p\eps^{2/3},\quad M\asymp_p\eps^{-1/3},\quad
 N\asymp_p\eps^{-4/3},\quad A=r\asymp_p\eps^{-1/3},\quad
 \tau\asymp_p\eps^{4/3},\\
 F(0)-\inf F\le\tfrac14,\qquad
 \sup_z\|D^{q+1}F(z)\|_{\op}\le b\ (1\le q\le p),\qquad
 u_N=0\ \Longrightarrow\ \|\nabla F(z)\|\ge32\eps.
 \end{gathered}
\end{equation}
The key cancellation is $\tau A^{q+1}r^{3-q}=\tau A^4$.
Thus raising the prescribed finite order does not change the accuracy
exponent. A scalar transition certificate and an exact block residual identity
prove the last implication in~\eqref{eq:overview-scales}; the terminal block
is handled with its boundary term intact.

\subsection{Repeated spectra and static tensors}
Take $K$ extra copies of the positive-semidefinite quadratic $Q_0$, keeping
the linear and nonlinear terms in the first copy only. The canonical matrix
is $\overline Q=I_{K+1}\otimes Q_0$, the dimension is
$d=(K+1)D_0$, and $D_0=NM$. Rotate the full construction by one Haar
$U\in O(d)$, giving a base $F_U$ and physical quadratic matrix
$Q_U=U\overline Q U^\top$. Denote its unit active readouts by
$b_s=U\bar b_s$, so the actual readouts have norm $A$ and equal $Ab_s$.
These operations preserve the analytic bounds in~\eqref{eq:overview-scales}.

Independently draw standard Gaussian tensors $Z_k\in\Sym^k(\R^d)$,
$k=3,4$, and define
\begin{equation}\label{eq:overview-mask}
 P_Z(y)=\frac{\sigma_3}{3}Z_3[y,y,y]
       +\frac{\sigma_4}{4}Z_4[y,y,y,y],\qquad
 \sigma_k=\frac{\delta_k}{a_k\sqrt d},\quad
 a_k=4\sqrt{k\log(1+4k)+2}.
\end{equation}
Let $\mathcal G_Z=\{\sigma_k\|Z_k\|_{\op}\le\delta_k,
\ k=3,4\}$. A sphere-net argument gives
$\Pr(\mathcal G_Z)\ge1-4e^{-4d}$. This is an event about the fixed instance,
not an event re-sampled at each oracle call.

\subsection{Unrestricted queries and a robust terminal certificate}
Set $R=\eps^{-1}$, $\eta_0=4\eps^2$,
$\delta_3=\eps^3/8$, and $\delta_4=\eps^4/8$. The final function is
\begin{equation}\label{eq:overview-final}
 f_{U,Z}(x)=(F_U+P_Z)(\rho_R(x))+\tfrac{\eta_0}{2}\|x\|^2,
 \qquad \rho_R(x)=\frac{x}{\sqrt{1+\|x\|^2/R^2}}.
\end{equation}
Condition its law on $\mathcal G_Z$ to obtain $\mathcal D_{p,\eps}$.
Theorem~\ref{ana:thm:confinement} proves membership in $\mathcal H_{p,d}$
pointwise on this event, with initial gap at most $1/3$. It also proves
\begin{equation}\label{eq:overview-certificate}
 |b_N^\top\rho_R(x)|\le h_{\mathrm{cert}}:=16\eps/b
 \quad\Longrightarrow\quad \|\nabla f_{U,Z}(x)\|>\eps.
\end{equation}
For $\|x\|\ge R$ the radial gradient alone exceeds $3\eps$. For
$\|x\|\le R$, the Jacobian's smallest singular value is at least
$1/(2\sqrt2)$, and the perturbation-gradient budget is
$\delta_3R^2+\delta_4R^3=\eps/4$. These facts preserve the block certificate.
Every finite original query corresponds to a bounded internal query, but the
algorithm itself is never restricted to a bounded domain.

\section{Information hiding under exact adaptive queries}\label{sec:information}
The lower bound combines a geometric statement with an exact Gaussian
comparison. Their roles are separate: the geometric bound is proved under
an auxiliary reference process, and the information bound transfers failure
to a static-function distribution.

\subsection{A scheduled bound on the whole query span}
Give active coordinate $(s,j)$ layer
$\ell(s,j)=(s-1)j_\star+\min\{j,j_\star\}$, and let $I_t$ span canonical
coordinates with layer at most $t$. Then
$\overline QI_t\subseteq I_{t+1}$, $c_0\in I_1$, and the readout
$\bar b_s$ is orthogonal to $I_t$ until $t=sj_\star$.
For the analysis, disclose $U|_{I_t}$ at time $t$ and supply one exact
quadratic matrix-vector product per query. These are strengthened observations,
not extra resources required by the original algorithm.

Let $L_{\mathrm{cand}}$ count the queries and one unqueried output.
For any allowed reference policy, deferred Haar exposure produces canonical
iid Gaussian starts $g_1,\ldots,g_{L_{\mathrm{cand}}}$ such that
\begin{equation}\label{eq:overview-cyclic}
 U^\top x_t\in I_t+\mathcal R,\qquad
 \mathcal R=\bigoplus_\lambda
       \Span\{P_\lambda g_1,\ldots,P_\lambda g_{L_{\mathrm{cand}}}\}.
\end{equation}
The starts may depend on $U$; only their Gaussian law relative to the fixed
canonical subspaces is used. Because each eigenspace has $K+1$ repeated
copies, a random-subspace angle estimate implies, with probability at least
$1-\zeta$,
\begin{equation}\label{eq:overview-angle}
 \|P_{W_s}b_s\|\le\theta\quad\text{for every }s,\qquad
 \theta^2=O\left(\frac{L_{\mathrm{cand}}+\log(D_0/\zeta)}{K}\right),
\end{equation}
where $W_s$ is the span of all internal queries made before readout $s$ is
disclosed. The same event controls the final output if its readout remains
undisclosed. Crucially,~\eqref{eq:overview-angle} bounds an entire span and
therefore already includes differences of nearly dependent queries.
Appendix~\ref{app:geometry} proves the statement, including its conditional
Haar law and the precise allowed side information.

\subsection{Censoring only Gaussian innovations}
Within the undisclosed quartic cores, the two nonlinear signal tensors are
\begin{equation}\label{eq:overview-signal}
 S_3=-120\tau A^3\sum_s b_s^{\otimes3},\qquad
 S_4=120\tau A^4\sum_s b_s^{\otimes4}.
\end{equation}
Write $S_k=w_k\sum_s b_s^{\otimes k}$ and
$\Theta_k=S_k+\sigma_kZ_k$. The strengthened experiment returns separate
channels $Y_{k,t}=\Theta_k[z_t^{\otimes(k-1)},\cdot]$, where
$z_t=\rho_R(x_t)$. Their values are recovered by contracting with $z_t/k$.
The quadratic part and already disclosed taper corrections complete the exact
value-and-gradient answer before a hidden readout leaves its core.

For degree $k$, let
\begin{equation}\label{eq:overview-observation-space}
 V_{k,t}=\Span\{\Sym(z_i^{\otimes(k-1)}\otimes v):i\le t,\ v\in\R^d\},
 \qquad \Delta V_{k,t}=V_{k,t}\cap V_{k,t-1}^\perp.
\end{equation}
Let $M_{k,t-1}$ be the minimum-norm tensor fitting past answers,
$L_{k,t}A=A[z_t^{\otimes(k-1)},\cdot]$, and
$B_{k,t}=L_{k,t}(I-P_{V_{k,t-1}})$. Conditional on $U$ and the history,
the true channel is Gaussian with mean
$L_{k,t}M_{k,t-1}+B_{k,t}S_k$ and covariance
$\sigma_k^2B_{k,t}B_{k,t}^*$. The reference keeps this covariance and the
old fitted mean, replacing only $S_k$ by its currently disclosed part.

\begin{proposition}[Exact adaptive comparison]\label{prop:main-kl}
Let $\Ptrue$ and $\Pref$ be the joint laws of the hidden rotation, the seed,
and the transcript in the true static-tensor and innovation-only reference
experiments, respectively. Then
\begin{equation}\label{eq:overview-kl}
 \KL(\Pref\Vert\Ptrue)=
 \sum_{k=3}^4\frac1{2\sigma_k^2}\sum_{t=1}^T
 \E_{\Pref}\|P_{\Delta V_{k,t}}S_{k,\hid}(t)\|_F^2.
\end{equation}
If~\eqref{eq:overview-angle} holds with failure probability at most $\zeta$,
then
\begin{equation}\label{eq:overview-kl-bound}
 \KL(\Pref\Vert\Ptrue)\le
 \frac{N^2}{2}\sum_{k=3}^4\frac{k w_k^2}{\sigma_k^2}
           \bigl(\theta^{2(k-1)}+\zeta\bigr).
\end{equation}
The identities permit singular measurements, repeated queries, and arbitrary
adaptive query scales.
\end{proposition}
\begin{proof}[Proof outline; full proof in Appendix~\ref{app:gaussian}]
At a common history the two Gaussian kernels have the same affine support.
Their mean difference is $B_{k,t}S_{k,\hid}(t)$, and
$B^*(BB^*)^\dagger B=P_{\range(B^*)}=P_{\Delta V_{k,t}}$ yields the
conditional KL. Inductive absolute continuity justifies evaluating the true
kernel on reference histories. A fully repeated measurement has $B=0$ and
returns its old answer exactly. For a unit $b$,
$\|P_{\mathcal V_k(W)}b^{\otimes k}\|_F^2
\le k\|P_Wb\|^{2(k-1)}$, where $\mathcal V_k(W)$ is the full observation
space over $W$. Orthogonality of the increments $\Delta V_{k,t}$ telescopes
the time sum before each disclosure. Cauchy--Schwarz over readouts gives
\eqref{eq:overview-kl-bound}, with no additional factor of $T$.
\end{proof}

The reference can be sampled using independent physical Gaussian innovation
seeds and visible readouts only. After fixing those seeds, it is an allowed
policy for the geometric argument. This establishes the required compatibility
between the two parts of the proof; it does not condition on the full static
tensors or give the algorithm the entire quadratic matrix.

\section{Proof of Theorem~\ref{thm:matching}}\label{sec:main-proof}
We assemble the scales here; Appendix~\ref{sec:assembly} records the full
experiment and quantifier conversion. Take
\begin{equation}\label{eq:overview-final-scales}
 T=\lfloor Nj_\star/2\rfloor,\qquad L_{\mathrm{cand}}=T+1,
 \qquad K=\lceil\eps^{-14}\rceil,\qquad\zeta=K^{-4}.
\end{equation}
For sufficiently small $\eps$, $T+1<Nj_\star$ and
$T=\Theta_p(\eps^{-5/3})$. By construction,
$w_3=O_p(\eps^{1/3})$, $w_4=O_p(1)$, and
$d=(K+1)NM=O_p(\eps^{-47/3})$. Substituting
$\sigma_k=\delta_k/(a_k\sqrt d)$ and~\eqref{eq:overview-angle} into
\eqref{eq:overview-kl-bound} gives
\begin{equation}\label{eq:overview-final-kl}
 \KL(\Pref\Vert\Ptrue)
 \le C_p\left(\frac{\eps^{-13}}K+
 \frac{\eps^{-52/3}}{K^2}+\frac{\eps^{-37/3}}{K^3}\right)
 =O_p(\eps).
\end{equation}
Also $\theta^2=O_p(\eps^{37/3})$ and
$R\theta=O_p(\eps^{31/6})=o(\eps)$. Thus on the reference geometric event,
all undisclosed readouts satisfy $|Ab_s^\top z_t|\le r=A$, and the final
output obeys~\eqref{eq:overview-certificate}.

Let $\mathcal B$ be the event that a hidden readout exits its core or that
the terminal readout exceeds $h_{\mathrm{cert}}$. It is measurable from the
joint rotation and transcript. The reference bound and Pinsker's inequality
give
\begin{equation}\label{eq:overview-bad-event}
 \Ptrue(\mathcal B)\le\zeta+
 \sqrt{\KL(\Pref\Vert\Ptrue)/2}.
\end{equation}
Using the same static instance and seed, couple the assembled true-channel
run and the original fixed-function run until their first hidden core exit.
Their answers agree exactly before that exit. Consequently, actual success
on $\mathcal G_Z$ implies $\mathcal B$ in the coupled channel run. Only now
condition on the tensor norm event:
\begin{equation}\label{eq:overview-success}
 \Pr_{f\sim\mathcal D_{p,\eps},\omega}
       \{\|\nabla f(x_{\mathrm{out}})\|\le\eps\}
 \le\frac{\zeta+\sqrt{\KL(\Pref\Vert\Ptrue)/2}}{1-4e^{-4d}}
 =O_p(\sqrt\eps)\le\tfrac13.
\end{equation}
No truncated Gaussian is used in the posterior calculation. The distribution
is fixed for all policies; the auxiliary geometric coupling may depend on the
policy, with a uniform probability bound.

For per-instance almost-sure query caps, truncate the policy after $T$ calls
on every possible answer history. On each admissible fixed function, this
changes behavior only on a null set of private seeds; integration over the
hard distribution preserves that fact. Averaging then rules out success
$2/3$ on every fixed function. Choosing $c_p$ below the asymptotic coefficient
in $T$ proves the lower bound.

For the upper bound, Theorem~2 of~\cite{carmon} applies with $L_1=L_3=1$.
To cover arbitrarily small actual initial gaps uniformly, simulate it on
$\widetilde f(x,t)=f(x)+(t-1)^2/2$ from $(0,0)$. Its gap lies in $[1/2,3/2]$,
its smoothness constants are unchanged, and a lifted query requires one
original query. Since
$\|\nabla\widetilde f(x,t)\|^2=\|\nabla f(x)\|^2+(t-1)^2$,
a lifted stationary point yields an original one. Appendix~\ref{sec:upper}
checks the accuracy condition and query accounting, giving the last
inequality in~\eqref{eq:main-rate}.\hfill$\square$

\section{Discussion and limitations}\label{sec:discussion}
The result separates higher-order smoothness from higher-order information:
under any fixed finite collection of normalized smoothness budgets, arbitrary
private randomization does not improve the dimension-uniform first-order
accuracy exponent beyond $5/3$. The mechanism is information-theoretic, not a
restriction to a particular iteration family or to finite-precision queries.
The general Gaussian comparison also isolates what must be controlled when
static perturbations are used in exact-oracle lower bounds: the new observation
subspace, not the size of a raw response.

Several quantitative questions remain. The upper-bound logarithm is not
removed. The dimension $O_p(\eps^{-47/3})$ is sufficient and deliberately
conservative; it is not a matching dimension-accuracy tradeoff. The constants
are not uniform in $p$, and the theorem does not address analytic classes with
one common bound at all orders. We normalize the smoothness budgets and do
not claim a complete parameter-sensitive phase diagram. Finally, the hard
distribution establishes worst-case query complexity, not a practical runtime
barrier or an empirical claim about typical optimization objectives.

\phantomsection
\addcontentsline{toc}{section}{References}
\begingroup
\small
\bibliographystyle{unsrt}
\bibliography{references}
\endgroup

\clearpage
\appendix
\section{The analytic block-chain construction}\label{app:analytic}
This appendix proves the deterministic properties used in
Section~\ref{sec:construction}. The block architecture follows~\cite{zhou};
the exact-core potential, its calibration, and the needed residual identities
are given explicitly here. The proof uses full multilinear operator bounds,
which imply the directional bounds in~\eqref{eq:function-class}.

\subsection{A smooth potential with an exact quartic core}
Fix an even function $\chi\in C^\infty(\R;[0,1])$ satisfying
$\chi(t)=1$ for $|t|\le1$ and $\chi(t)=0$ for $|t|\ge2$.
Such a cutoff can, for example, be constructed from
$b(t)=\exp(-1/t)\mathbf 1_{\{t>0\}}$ and a smooth step function.
Define
\[
 W(t)=\chi(t)+(1-\chi(t))t^{-2}\quad(t\ne0),\qquad W(0)=1.
\]
The apparently singular second summand vanishes identically near zero.
For $r\ge1$, let
\begin{equation}\label{ana:eq:primitive}
 \phi_r(u)=120\int_1^u v^2(v-1)W(v/r)\,dv.
\end{equation}

\begin{lemma}[Exact core and derivative bounds]\label{ana:lem:primitive}
The function $\phi_r$ is $C^\infty$ and satisfies
\begin{align}
 \phi_r(u)&=30u^4-40u^3+10 &&(|u|\le r),\label{ana:eq:core}\\
 \phi_r'(u)&=120r^2(u-1) &&(|u|\ge2r),\label{ana:eq:tail}\\
 \inf_u\phi_r(u)&=\phi_r(1)=0,\qquad
 \phi_r(0)=10,\qquad \phi_r'(0)=\phi_r'(1)=0.
\end{align}
For every integer $q\ge1$, a finite constant $C_q$ satisfies
\begin{equation}\label{ana:eq:primitive-derivatives}
 \|\phi_r^{(q+1)}\|_\infty
 =\Lip(\phi_r^{(q)})\le C_q r^{3-q}.
\end{equation}
Moreover,
\begin{equation}\label{ana:eq:radial-primitive}
 u\phi_r'(u)\ge-\frac{405}{32}\qquad(u\in\R).
\end{equation}
\end{lemma}
\begin{proof}
The cutoff is constant near its gluing points, so all formulas define smooth
functions. Since $r\ge1$, integration between $1$ and any $u\in[-r,r]$
uses $W=1$, proving \eqref{ana:eq:core}. Formula \eqref{ana:eq:tail} follows from
$W(t)=t^{-2}$ outside $[-2,2]$. The function $W$ is strictly positive;
therefore $\phi_r'(u)$ has the sign of $u-1$, except for its zero at $u=0$.
This proves the minimum and endpoint assertions.

Put $v=u/r$, $A_3(v)=v^3W(v)$, and $A_2(v)=v^2W(v)$.
Then
\[
 \phi_r'(u)=120\bigl[r^3A_3(v)-r^2A_2(v)\bigr].
\]
Both $A_3$ and $A_2$ are smooth; outside $[-2,2]$ they equal $v$ and $1$,
respectively. Thus their derivatives of every positive order are bounded.
One may take
\[
 C_q=120\bigl(\|A_3^{(q)}\|_\infty+
                   \|A_2^{(q)}\|_\infty\bigr).
\]
The factor $r^{2-q}$ is at most $r^{3-q}$ because $r\ge1$.
Finally $u\phi_r'(u)\ge0$ unless $0<u<1$. On that interval the core is exact,
and $120u^3(1-u)$ is maximized at $u=3/4$, with value $405/32$.
\end{proof}

\subsection{The scalar and block gradient certificates}
Only the parameter-one scalar transition certificate is needed for the
normalized block chain. The following direct proof establishes that certificate
without making any claim about the more general parameter-$\mu$ statement.

\begin{lemma}[Scalar transition]\label{ana:lem:scalar}
Let $N\ge1$, let $u_0=1$, $u_{N+1}=0$, and suppose $u_N<1/10$.
For $s\in[N]$, set
\[
 d_s=2u_s-u_{s-1}-u_{s+1}+\phi_r'(u_s).
\]
Then $\|d\|_2\ge1/4$. In particular this holds if $u_N=0$.
\end{lemma}
\begin{proof}
Suppose $\|d\|_2<1/4$, so every $|d_s|<1/4$.
Since $W$ is a convex combination of $1$ and $t^{-2}$,
\[
 u^2W(u/r)\ge\min\{u^2,r^2\}.
\]
If the minimum of $u_0,\ldots,u_{N+1}$ were at most $-1/10$, it would
occur at an index $s\in[N]$. Its discrete Laplacian would be nonpositive,
whereas
\[
 \phi_r'(u_s)
 =-120u_s^2(1-u_s)W(u_s/r)\le-120\cdot\frac1{100}\cdot\frac{11}{10}
 =-1.32.
\]
This contradicts $|d_s|<1/4$. All coordinates are consequently greater than
$-1/10$.

If $u_s\in[1/10,9/10]$, the exact core and the preceding lower bound on
neighbors give
\[
 d_s\le\frac15-u_s\bigl(120u_s(1-u_s)-2\bigr)
 \le\frac15-\frac{44}{5}u_s\le-0.68.
\]
Here $u_s(1-u_s)\ge9/100$ on the indicated interval. Thus every coordinate
$u_s$, $s\in[N]$, is either below $1/10$ or above $9/10$.
There is a first index $j$ below $1/10$, because $u_N<1/10$.
Its predecessor is greater than $9/10$, its successor is greater than
$-1/10$, and $\phi_r'(u_j)\le0$. Therefore
\[
 d_j<\frac15-\frac9{10}+\frac1{10}=-\frac35,
\]
again a contradiction.
\end{proof}

Let $H\in\R^{M\times M}$ be positive definite, $H\succeq\gamma I$,
let $\|e\|=1$, and let $a\ne0$. Write $A=\|a\|$ and assume
\[
 \alpha=a^\top H^{-1}e>0,\qquad
 \beta=a^\top H^{-1}a,\qquad \rho=e^\top H^{-1}e.
\]
Set
\[
 \lambda=\alpha^{-1},\qquad \tau=\beta^{-1},\qquad
 \kappa=\rho/\alpha^2.
\]
For $z=(z_1,\ldots,z_N)\in(\R^M)^N$, define $u_s=a^\top z_s$,
$u_0=1$, and
\begin{equation}\label{ana:eq:block}
 F(z)=\sum_{s=1}^N\left[
 \frac12z_s^\top Hz_s-\lambda u_{s-1}e^\top z_s
 +\frac\kappa2u_s^2+\tau\phi_r(u_s)\right].
\end{equation}
Write its quadratic and affine part as
\begin{equation}\label{ana:eq:Q-c}
 F(z)=\frac12z^\top Qz+c^\top z+\tau\sum_{s=1}^N\phi_r(u_s),
 \qquad c=(-\lambda e,0,\ldots,0).
\end{equation}

\begin{proposition}[Block certificate and analytic bounds]\label{ana:prop:block}
If $u_N=0$, then
\begin{equation}\label{ana:eq:block-cert}
 \|\nabla F(z)\|
 \ge\frac{\sqrt\gamma}
 {4\left(\sqrt\beta+\beta\sqrt\rho/\alpha\right)}.
\end{equation}
The matrix $Q$ is positive semidefinite, $c\in\operatorname{range}(Q)$, and
\begin{equation}\label{ana:eq:pseudoinverse}
 c^\top Q^\dagger c\le\kappa.
\end{equation}
Furthermore,
\begin{align}
 F(0)-\inf F&\le\overline\Delta:=10N\tau+\kappa/2,\label{ana:eq:gap}\\
 \|Q\|&\le L_Q:=\|H\|+2\lambda A+\kappa A^2,\label{ana:eq:Qnorm}\\
 \|D^2F\|_\infty&\le L_Q+C_1\tau A^2r^2,\label{ana:eq:hess}\\
 \|D^{q+1}F\|_\infty&\le C_q\tau A^{q+1}r^{3-q}\quad(q\ge2),\label{ana:eq:high}\\
 \inner{z}{\nabla F(z)}&\ge-\kappa/4-\frac{405}{32}N\tau
 \ge-\frac{81}{64}\overline\Delta.\label{ana:eq:radial-F}
\end{align}
\end{proposition}
\begin{proof}
Let $g_s$ be the $s$th gradient block and put
\[
 w_{\rm out}=H^{-1}a/\beta,\qquad
 w_{\rm in}=H^{-1}(e-\alpha a/\beta).
\]
For $s<N$, direct substitution gives the algebraic identity
\[
 w_{\rm out}^\top g_s+\lambda w_{\rm in}^\top g_{s+1}
 =\tau(2u_s-u_{s-1}-u_{s+1}+\phi_r'(u_s))=\tau d_s.
\]
With $u_{N+1}=0$, the exact terminal identity is
\[
 w_{\rm out}^\top g_N=\tau d_N+(\kappa-\tau)u_N.
\]
The additional term vanishes under the hypothesis $u_N=0$. Hence
\[
 \tau\|d\|_2\le
 (\|w_{\rm out}\|+\lambda\|w_{\rm in}\|)\|g\|_2.
\]
Also
\[
 \|w_{\rm out}\|\le(\gamma\beta)^{-1/2},\qquad
 \lambda\|w_{\rm in}\|
 \le\frac{\sqrt\rho}{\alpha\sqrt\gamma}.
\]
Lemma~\ref{ana:lem:scalar} proves \eqref{ana:eq:block-cert}.

Completing the square phase by phase yields
\[
 \frac12z_s^\top Hz_s-\lambda u_{s-1}e^\top z_s
 +\frac\kappa2u_s^2
 =\frac12\|z_s-\lambda u_{s-1}H^{-1}e\|_H^2
 +\frac\kappa2(u_s^2-u_{s-1}^2).
\]
Consequently the quadratic and affine part is bounded below by $-\kappa/2$.
A quadratic polynomial bounded below has $Q\succeq0$,
$c\in\operatorname{range}(Q)$, and minimum $-c^\top Q^\dagger c/2$.
This proves \eqref{ana:eq:pseudoinverse}. Since $\phi_r\ge0$ and
$F(0)=10N\tau$, it also proves \eqref{ana:eq:gap}.

The quadratic-form estimate \eqref{ana:eq:Qnorm} follows by bounding the
neighboring block cross terms with $2xy\le x^2+y^2$.
The readouts belong to orthogonal blocks. For arbitrary unit vectors
$v^{(0)},\ldots,v^{(q)}$ in the product space,
\[
 \sum_s\prod_{j=0}^q|a^\top v_s^{(j)}|
 \le\prod_{j=0}^q\left(\sum_s|a^\top v_s^{(j)}|^2\right)^{1/2}
 \le A^{q+1}.
\]
The first inequality is generalized H\"older with exponent $q+1$ in every
factor, followed by the monotonicity
$\|\cdot\|_{q+1}\le\|\cdot\|_2$ of finite-dimensional $\ell_r$ norms.
Together with Lemma~\ref{ana:lem:primitive}, this gives \eqref{ana:eq:hess}
and \eqref{ana:eq:high}.

Finally the minimum of $z^\top Qz+c^\top z$ is
$-c^\top Q^\dagger c/4\ge-\kappa/4$.
Adding $\tau\sum_su_s\phi_r'(u_s)$ proves the first radial inequality.
The second follows since $(405/32)/10=81/64$ and
$(\kappa/4)/(\kappa/2)=1/2\le81/64$.
\end{proof}

\subsection{Spectral decoys and the exact polynomial representation}
For any integer $K\ge1$, take $K+1$ copies of the quadratic matrix $Q$ and
activate the nonlinearity in one copy:
\begin{equation}\label{ana:eq:decoys}
 F_{\rm ext}(z^{(1)},\ldots,z^{(K+1)})
 =F(z^{(1)})+\frac12\sum_{j=2}^{K+1}(z^{(j)})^\top Qz^{(j)}.
\end{equation}
An arbitrary fixed orthogonal rotation of this objective is allowed.
Its quadratic matrix is an orthogonal conjugate of $I_{K+1}\otimes Q$, so every
eigenvalue of $Q$ has multiplicity multiplied by $K+1$.
Its linear vector is the correspondingly rotated $(c,0,\ldots,0)$.
We use $a_s$ for the rotated active readouts; they remain pairwise orthogonal,
with $\|a_s\|=A$.

All gap and radial estimates in Proposition~\ref{ana:prop:block} hold unchanged
for $F_{\rm ext}$. The zero-tail certificate also holds unchanged: its active
gradient component is $\nabla F$, and the decoy gradient components are
orthogonal to it. Because $\phi_r''(0)=0$, $D^2F(0)=Q$, so if
$\|D^2F\|_\infty\le B_2$, then $\|Q\|\le B_2$. Consequently
\begin{equation}\label{ana:eq:decoy-smooth}
 \|D^jF_{\rm ext}\|_\infty\le B_j\quad(2\le j\le p+1)
\end{equation}
whenever these bounds hold for $F$. None of these assertions depends on $K$.
The same conclusions permit extra zero quadratic directions.

Define the scalar correction
\[
 \mathcal C_r(u)=\phi_r(u)-(30u^4-40u^3+10).
\]
It and all its derivatives vanish on $[-r,r]$. Thus
\begin{align}\label{ana:eq:polynomial-plus-correction}
 F_{\rm ext}(y)
 &=\frac12y^\top Q_{\rm ext}y+c_{\rm ext}^\top y
 +\tau\sum_{s=1}^N[30(a_s^\top y)^4-40(a_s^\top y)^3+10]
 +\tau\sum_{s=1}^N\mathcal C_r(a_s^\top y).
\end{align}
If the undisclosed readouts satisfy $|a_s^\top y|\le r$, the last sum and
its gradient involve only the disclosed readouts. In particular, only the
\emph{undisclosed readouts} must lie in the exact polynomial core; the query
norm and already disclosed readouts need not do so. Once a readout is known,
its correction in the exact value-gradient oracle is explicitly computable.

\subsection{A simultaneous all-order calibration}
Here is one explicit scaling that leaves room for the later confinement map.
Fix an integer $p\ge3$ and a desired small derivative budget $b\in(0,1]$.
Let $\mathsf P_M$ be the path Laplacian with endpoint diagonal entries $1$,
interior diagonal entries $2$, and neighboring off-diagonal entries $-1$.

\begin{lemma}[Elementary path estimates]\label{ana:lem:green}
Let $0<h\le1/4$, $M=\lceil h^{-1}\rceil$, and
$H=L(h^2I+\mathsf P_M)$, $\gamma=Lh^2$.
Then
\[
 \|H\|\le4L,
 \qquad
 \frac1{9Lh}<(H^{-1})_{ij}\le\frac{16}{Lh}\quad(i,j\in[M]).
\]
Let $S=\{\lceil M/2\rceil,\ldots,M\}$, $s=|S|$, and
$a=A\mathbf1_S/\sqrt s$, $e=e_1$. The parameters of
Proposition~\ref{ana:prop:block} obey
\begin{align}
 \frac{A^2}{18\gamma}\le\beta\le\frac{A^2}{\gamma},\qquad
 \kappa\le\frac{2592\gamma}{A^2},\qquad
 2\lambda A\le18\sqrt2\gamma^{3/4}L^{1/4},\label{ana:eq:green-consequences}\\
 \frac{\sqrt\gamma}{4(\sqrt\beta+\beta\sqrt\rho/\alpha)}
 \ge\frac{\gamma}{208 A}.\label{ana:eq:green-gradient}
\end{align}
\end{lemma}
\begin{proof}
The path Laplacian has eigenvalues
$2-2\cos(k\pi/M)$, $k=0,\ldots,M-1$. Hence
\[
 \lambda_{\max}(\mathsf P_M)
 =2+2\cos(\pi/M)
 =4-4\sin^2\!\frac{\pi}{2M}.
\]
Because $M=\lceil h^{-1}\rceil$ and $h\le1/4$, we have $M\ge4$ and
$h\le2/M$.  Concavity of $\sin$ on $[0,\pi/2]$ gives
$\sin(\pi/(2M))\ge1/M$, so
$h^2\le4/M^2\le4\sin^2(\pi/(2M))$. Therefore
\[
 \|H\|=L\bigl(h^2+\lambda_{\max}(\mathsf P_M)\bigr)\le4L.
\]
For the inverse estimates, put
$t=(2+h^2-h\sqrt{h^2+4})/2\in(0,1)$.
The inverse of $h^2I+\mathsf P_M$ has entries
\[
 \frac{t^{|i-j|+1}(1+t^{2(i\wedge j)-1})
 (1+t^{2(M-(i\vee j))+1})}{(1-t^2)(1-t^{2M})}.
\]
This formula follows by solving the second-order homogeneous recurrence and
the two endpoint conditions (or by direct multiplication).
Indeed $t/(1-t^2)=1/(h\sqrt{h^2+4})$ and
$-\log t=2\operatorname{arsinh}(h/2)\in[h/2,h]$.
The two factors in the numerator lie between $1$ and $2$,
$\sqrt{h^2+4}<3$, and $1-t^{2M}\ge1-e^{-1}>1/2$.
These inequalities give the stronger bounds
$e^{-h|i-j|}/(3Lh)$ and $16e^{-h|i-j|/2}/(Lh)$.
Since $h(M-1)<1$ and $e<3$, the displayed uniform bounds follow.

The rounding bounds give $s\ge1/(2h)$. Hence
\[
 \alpha\ge\frac{A\sqrt s}{9Lh},\qquad
 \rho\le\frac{16}{Lh},\qquad
 \beta\ge\frac{A^2s}{9Lh}\ge\frac{A^2}{18\gamma}.
\]
The upper bound on $\beta$ uses $H^{-1}\preceq\gamma^{-1}I$.
Substitution proves \eqref{ana:eq:green-consequences}.
Finally
\[
 \sqrt\beta\le\frac A{\sqrt\gamma},\qquad
 \frac{\beta\sqrt\rho}{\alpha}
 \le\frac{36\sqrt2 A}{\sqrt\gamma},
\]
and $1+36\sqrt2<52$ proves \eqref{ana:eq:green-gradient}.
\end{proof}

\begin{proposition}[Uniform smoothness and the intended horizon]\label{ana:prop:calibration}
Fix $G_0=8192$ and $L=b/16$. Let
$C_*^{\rm prim}=\max_{1\le q\le p}C_q$, and choose $a_0>0$ so small that
\[
 18G_0 C_*^{\rm prim} a_0^3\le b/4.
\]
For all sufficiently small $\eps>0$, set
\begin{align}\label{ana:eq:calibration}
 A=r=a_0\eps^{-1/3},\qquad
 \gamma=G_0a_0\eps^{2/3},\qquad
 h=\sqrt{\gamma/L},\qquad M=\lceil h^{-1}\rceil,\\
 a=A\mathbf1_S/\sqrt{|S|},\qquad
 N=\left\lfloor\frac{a_0}{1440G_0}\eps^{-4/3}\right\rfloor,
 \nonumber
\end{align}
and use the normalized $\alpha,\beta,\lambda,\tau,\kappa$ above.
Then $r\ge1$, $h\le1/4$, $N\ge1$, and
\begin{align}
 \sup_z\|D^jF(z)\|_{\op}&\le b&& (2\le j\le p+1),\label{ana:eq:small-base-budgets}\\
 \overline\Delta=10N\tau+\kappa/2&\le1/4,\label{ana:eq:small-base-gap}\\
 a^\top z_N=0&\Longrightarrow\|\nabla F(z)\|\ge32\eps.\label{ana:eq:large-base-cert}
\end{align}
In addition, $N=\Theta_{p,b}(\eps^{-4/3})$,
$M=\Theta_{p,b}(\eps^{-1/3})$, and $NM=\Theta_{p,b}(\eps^{-5/3})$.
The physical core width is exactly $r/A=1$.
\end{proposition}
\begin{proof}
All integer and small-parameter requirements hold by reducing a fixed positive
threshold for $\eps$. Since $\tau\le18\gamma/A^2$, for every $q\in[p]$,
\[
 C_q\tau A^{q+1}r^{3-q}
 \le18C_q\gamma A^2
 =18C_qG_0a_0^3\le b/4.
\]
For $q=1$ this bounds the nonlinear Hessian; for higher $q$ it bounds the
entire derivative of order $q+1$. The quadratic bound is
\[
 L_Q\le4L+2593\gamma+18\sqrt2\gamma^{3/4}L^{1/4}.
\]
The first term is $b/4$ and the last two tend to zero, so $L_Q\le b/2$
for sufficiently small $\eps$. This proves \eqref{ana:eq:small-base-budgets}.
The phase-count choice gives
\[
 10N\tau\le\frac18,\qquad
 \frac\kappa2\le\frac{1296G_0}{a_0}\eps^{4/3}\le\frac18
\]
after reducing the threshold again. Finally the zero-tail certificate is at
least $(G_0/208)\eps>32\eps$. The integer asymptotics follow directly from
\eqref{ana:eq:calibration}.
\end{proof}

The horizon $NM$ is converted to a randomized lower bound by the geometric and Gaussian arguments in Appendices~\ref{app:geometry} and~\ref{app:gaussian}.

\section{Global admissibility and static tensor perturbations}\label{app:global}
We first prove deterministic confinement estimates for arbitrary bounded
third- and fourth-order perturbation tensors. We then verify the norm event
for the Gaussian tensors used in the hard distribution.

\subsection{Static polynomial masks and dimension-free confinement}
Let $T_3$ and $T_4$ be arbitrary symmetric multilinear forms satisfying
$\|T_k\|_{\op}\le\delta_k$, and define the static potential
\begin{equation}\label{ana:eq:mask}
 P(y)=\frac13T_3[y,y,y]+\frac14T_4[y,y,y,y].
\end{equation}
For a radius $R>0$, put
\[
 \nu=\delta_3R^2+\delta_4R^3.
\]
For $\|y\|\le R$, direct differentiation gives
\begin{align}
 \|\nabla P(y)\|&\le\nu,\qquad
 |P(y)|\le\frac{R\nu}{3},\qquad
 |\inner{y}{\nabla P(y)}|\le R\nu,\label{ana:eq:mask-basic}\\
 \|D^jP(y)\|_{\op}&\le
 \sum_{\substack{k\in\{3,4\}\\ k\ge j}}
 \frac{(k-1)!}{(k-j)!}\delta_k R^{k-j}\quad(j\ge1).
 \label{ana:eq:mask-deriv}
\end{align}
These estimates are deterministic. For the Gaussian construction they
apply on the norm event proved in Lemma~\ref{gau:lem:noise-norm}.

Define
\begin{equation}\label{ana:eq:rho}
 \rho_R(x)=\frac{x}{\sqrt{1+\|x\|^2/R^2}},\qquad
 G(x)=(F_{\rm ext}+P)(\rho_R(x))+\frac{\eta_0}{2}\|x\|^2.
\end{equation}

\begin{lemma}[Derivatives of the radial map and composition]\label{ana:lem:composition}
The map $\rho_R$ has image the open ball of radius $R$.
For each $j\ge1$ there is a dimension-free constant $d_j$ such that
\[
 \|D^j\rho_R(x)\|_{\op}\le d_jR^{1-j}\quad(x\in\R^d),\qquad d_1=1.
\]
If $t=\|x\|$ and $a_x=(1+t^2/R^2)^{-1/2}$, its Jacobian eigenvalues are
$a_x$ in tangential directions and $a_x^3$ in the radial direction.
In particular, for $t\le R$ its smallest singular value is at least
$1/(2\sqrt2)$.

Suppose $\|D^jF_{\rm ext}\|_\infty\le B_j$ for $2\le j\le p+1$,
and let $c_{\rm ext}=\nabla F_{\rm ext}(0)$.
There are constants $K_q\ge1$, depending only on $q$, such that
\begin{equation}\label{ana:eq:composition-budget}
 \|D^{q+1}G\|_\infty
 \le K_q\left[
 (\|c_{\rm ext}\|+\nu)R^{-q}
 +\sum_{j=2}^{q+1}B_j R^{j-q-1}\right]
 +\eta_0\mathbf1_{\{q=1\}}
 \quad(1\le q\le p).
\end{equation}
\end{lemma}
\begin{proof}
The Jacobian assertions follow by differentiation. The scaling identity
$\rho_R(x)=R\rho_1(x/R)$ reduces higher derivative bounds to $R=1$.
Differentiating $x(1+\|x\|^2)^{-1/2}$ yields a finite sum, with coefficients
depending only on the differentiation order, of products of inner products,
copies of $x$, and negative powers of $1+\|x\|^2$.
For unit input vectors each such term is bounded by a constant depending only
on that order: every power of $\|x\|$ is dominated by the corresponding
denominator power. This proves the dimension-free $d_j$ bounds.

For $n=q+1$, the multilinear chain rule is indexed by partitions of $[n]$.
A term with $j$ blocks contains $D^jF_{\rm ext}$ and a product of derivatives
of $\rho_R$ bounded by a constant times $R^{j-n}$.
For $j\ge2$ this gives $B_jR^{j-n}$. For $j=1$, use
\[
 \|\nabla F_{\rm ext}(y)\|\le\|c_{\rm ext}\|+B_2R
 \quad(\|y\|\le R).
\]
Its contribution is bounded by
$\|c_{\rm ext}\|R^{1-n}+B_2R^{2-n}$, with the latter absorbed in the
$j=2$ term. The same chain rule and \eqref{ana:eq:mask-deriv} give
\[
 \|D^n(P\circ\rho_R)\|_\infty
 \le C_n\bigl(\delta_3R^{3-n}+\delta_4R^{4-n}\bigr)
 =C_n\nu R^{1-n}.
\]
There are finitely many partitions depending only on $n$, so their constants
can be absorbed into $K_q$. The final quadratic term contributes only when
$q=1$.
\end{proof}

\begin{theorem}[Confinement, smoothness, gap, and robust stationarity]\label{ana:thm:confinement}
Fix $p\ge3$ and $C_R\ge1$. Let $K_*\ge1$ dominate the constants $K_q$,
$q\in[p]$, of Lemma~\ref{ana:lem:composition}, and take
\[
 0<b\le\frac1{16pK_*}.
\]
Use Proposition~\ref{ana:prop:calibration}, followed by any number of spectral
decoy copies and any fixed orthogonal rotation. For all sufficiently small
$\eps>0$, let
\begin{equation}\label{ana:eq:confinement-parameters}
 R=\frac{C_R}{\eps},\qquad
 \eta_0=\frac{4\eps}{R},\qquad
 \nu=\delta_3R^2+\delta_4R^3\le\frac{\eps}{4C_R}.
\end{equation}
Then $G$ from \eqref{ana:eq:rho} is $C^\infty$ and satisfies
\begin{align}
 G(0)-\inf G&\le\frac13<1,\label{ana:eq:G-gap}\\
 \Lip(D^qG)&\le1\quad(1\le q\le p),\label{ana:eq:G-budgets}\\
 \|x\|\ge R&\Longrightarrow\|\nabla G(x)\|>3\eps.
 \label{ana:eq:outside}
\end{align}
Let $a_N$ be the rotated final active readout, $\|a_N\|=A$.
At every $x$ such that $\|x\|\le R$ and, with $y=\rho_R(x)$,
\begin{equation}\label{ana:eq:approx-tail}
 |a_N^\top y|\le\frac{16A\eps}{b},
\end{equation}
one also has $\|\nabla G(x)\|>\eps$.
Consequently every $\eps$-stationary point lies in $\|x\|<R$ and violates
the hidden-tail condition \eqref{ana:eq:approx-tail}.
\end{theorem}
\begin{proof}
Because the added quadratic is nonnegative and $|P(y)|\le R\nu/3\le1/12$
on the image of $\rho_R$,
\[
 G(0)-\inf G\le\overline\Delta+\frac1{12}
 \le\frac14+\frac1{12}=\frac13.
\]
For smoothness, $B_j\le b$ in \eqref{ana:eq:composition-budget}.
The standard descent inequality for a bounded-below $b$-smooth function gives
\[
 \|c_{\rm ext}\|^2\le2b(F_{\rm ext}(0)-\inf F_{\rm ext})\le b/2.
\]
In particular $\|c_{\rm ext}\|\le1$. Since $R\ge1$,
\[
 \sum_{j=2}^{q+1} B_jR^{j-q-1}\le p b.
\]
Taking additionally $\eps\le\min\{1/8,1/(16K_*)\}$, the right side of
\eqref{ana:eq:composition-budget} is bounded by
\[
 K_*p b+K_*\eps+\frac{K_*\eps^2}{4}+4\eps^2<1.
\]
Thus all full operator derivative bounds, and hence all asserted Lipschitz
bounds, hold.

Put $V=F_{\rm ext}+P$.
By \eqref{ana:eq:radial-F} and \eqref{ana:eq:mask-basic}, for $\|y\|\le R$,
\[
 \inner{y}{\nabla V(y)}\ge-C_0,\qquad
 C_0=\frac{81}{64}\overline\Delta+R\nu
 \le\frac{81}{256}+\frac14=\frac{145}{256}.
\]
For $t=\|x\|\ge R$ and $a_x=(1+t^2/R^2)^{-1/2}$,
\[
 \inner{\frac{x}{t}}{\nabla G(x)}
 =\frac{a_x^2}{t}\inner{y}{\nabla V(y)}+\eta_0t
 \ge-\frac{C_0}{2R}+4\eps
 \ge\left(4-\frac{145}{512C_R}\right)\eps>3\eps.
\]
This proves \eqref{ana:eq:outside} without any assumption on hidden readouts.

For the interior statement, set
\[
 y_0=y-\frac{a_N^\top y}{A^2}a_N.
\]
This point has final readout zero and
$\|y-y_0\|=|a_N^\top y|/A$.
The $b$-Lipschitz gradient and the $32\eps$ zero-tail certificate give
\[
 \|\nabla F_{\rm ext}(y)\|
 \ge32\eps-b\frac{|a_N^\top y|}{A}\ge16\eps.
\]
The noise gradient has norm at most $\nu\le\eps/4$.
Since $\|x\|\le R$, the minimum singular value of $D\rho_R(x)$ is
at least $1/(2\sqrt2)$ and $\eta_0\|x\|\le4\eps$. Therefore
\[
 \|\nabla G(x)\|
 \ge\frac{16\eps-\nu}{2\sqrt2}-4\eps
 \ge\left(\frac{63}{8\sqrt2}-4\right)\eps>\eps.
\]
\end{proof}

\paragraph{Exact value-gradient simulation.}
At every finite query $x$, set $y=\rho_R(x)$; always $\|y\|<R$.
An exact value-gradient oracle for $V=F_{\rm ext}+P$ at $y$ determines
\[
 G(x)=V(y)+\eta_0\|x\|^2/2,\qquad
 \nabla G(x)=D\rho_R(x)^\top\nabla V(y)+\eta_0x.
\]
Conversely $D\rho_R(x)$ is invertible at every finite $x$, so the same
formulas recover the value and gradient of $V$ at $y$ from an exact oracle for
$G$ at $x$. These transformations are public, deterministic, and preserve
adaptive exact-oracle information. Combined with
\eqref{ana:eq:polynomial-plus-correction}, they require an exact polynomial core
only for undisclosed readouts at the transformed query $y$.

\subsection{Gaussian norm event and noise budgets}
For a tensor $Z\in\Sym^k(\R^d)$ use the full multilinear operator norm
\[
 \|Z\|_{\op}
 =\sup_{\|v_1\|=\cdots=\|v_k\|=1}|Z[v_1,\ldots,v_k]|.
\]
Define
\begin{equation}\label{gau:eq:gaussian-constant}
 a_k=4\sqrt{k\log(1+4k)+2},\qquad
 \sigma_k=\frac{\delta_k}{a_k\sqrt d}.
\end{equation}

\begin{lemma}[A fixed high-probability norm event]\label{gau:lem:noise-norm}
The event
\[
 \mathcal G_Z=\bigcap_{k=3}^4
 \{\sigma_k\|Z_k\|_{\op}\leq\delta_k\}
\]
has probability at least $1-4e^{-4d}$.  On this event the static polynomial
\[
 P(z)=\sum_{k=3}^4\frac{\sigma_k}{k}Z_k[z^{\otimes k}]
\]
satisfies, for $\|z\|\leq R$,
\begin{align}
 |P(z)|&\leq\frac{\delta_3R^3}{3}
               +\frac{\delta_4R^4}{4},\label{gau:eq:noise-value}\\
 \|\nabla P(z)\|&\leq\delta_3R^2+\delta_4R^3.\label{gau:eq:noise-grad}
\end{align}
For $j\geq1$,
\begin{equation}\label{gau:eq:noise-derivatives}
 \|D^jP(z)\|_{\op}
 \leq\sum_{k\in\{3,4\}:k\geq j}
   \frac{(k-1)!}{(k-j)!}\delta_k R^{k-j}.
\end{equation}
\end{lemma}

\begin{proof}
Take a $(2k)^{-1}$-net of the unit sphere with at most $(1+4k)^d$ points.
The product of $k$ such nets has at most $(1+4k)^{kd}$ elements.
Multilinearity and a telescoping approximation in the $k$ arguments give
$\|Z\|_{\op}\leq2\max_{\mathrm{net}}|Z[v_1,\ldots,v_k]|$.
For fixed unit arguments, this evaluation is Gaussian with variance
$\|\Sym(v_1\otimes\cdots\otimes v_k)\|_F^2\leq1$.
A union bound therefore gives
\[
 \Pr\{\|Z_k\|_{\op}>a_k\sqrt d\}
 \leq2\exp\{kd\log(1+4k)-a_k^2d/8\}
 \leq2e^{-4d}.
\]
The two-degree union bound proves the probability estimate.  Direct
differentiation of the homogeneous polynomials proves the remaining claims.
\end{proof}

One convenient choice, for a desired perturbation-gradient budget
$\mu\varepsilon$, is
\begin{equation}\label{gau:eq:delta-choice}
 \delta_3=\frac{\mu\varepsilon}{2R^2},\qquad
 \delta_4=\frac{\mu\varepsilon}{2R^3}.
\end{equation}
Then \eqref{gau:eq:noise-grad} is at most $\mu\varepsilon$ and
\eqref{gau:eq:noise-value} is at most $7\mu\varepsilon R/24$.
For
\[
 \rho_R(x)=\frac{x}{\sqrt{1+\|x\|^2/R^2}},
\]
the bounds $\|\rho_R(x)\|\leq R$ and
$\|D^j\rho_R(x)\|_{\op}\leq C_jR^{1-j}$ imply, by the finite-order
chain rule,
\begin{equation}\label{gau:eq:noise-composition}
 \sup_x\|D^j(P\circ\rho_R)(x)\|_{\op}
 \leq C_j\bigl(\delta_3R^{3-j}+\delta_4R^{4-j}\bigr)
 \leq C_j\mu\varepsilon R^{1-j},\qquad j\geq1.
\end{equation}
Thus every prescribed finite derivative budget can be met by a fixed choice
of constants.  In particular, $R\asymp\varepsilon^{-1}$ gives
$\delta_3\asymp\varepsilon^3$ and
$\delta_4\asymp\varepsilon^4$.

\section{Scheduled Haar exposure and whole-span control}\label{app:geometry}
The statements in this appendix permit arbitrary measurable query policies.
The only restriction concerns the source of their side information: after
fixing independent physical advice, reference responses may depend on the
rotation only through scheduled columns and counted query pairs.

\subsection{Exact model and deterministic canonical data}
Fix a known real symmetric matrix \(Q_0\in\R^{D_0\times D_0}\),
an integer \(K\ge0\), and a horizon \(L\ge1\).
Set
\[
 d=(K+1)D_0,\qquad
 \overline Q=\diag(Q_0,\ldots,Q_0)\in\R^{d\times d}.
\]
The canonical active subspace is the first copy
\[
 A_0=\R^{D_0}\oplus0\oplus\cdots\oplus0\subseteq\R^d.
\]
It is \(\overline Q\)-invariant.
Let \(J\) be the number of distinct eigenvalues of \(Q_0\).
For each such eigenvalue \(\lambda\), its multiplicity in \(Q_0\)
is denoted \(m_\lambda\ge1\), and its multiplicity in \(\overline Q\)
is \((K+1)m_\lambda\).
The corresponding fixed canonical spectral projector is \(P_\lambda\).

Fix deterministic nested subspaces
\[
 I_1\subseteq I_2\subseteq\cdots\subseteq I_{L+1}\subseteq A_0
\]
with known bases and the one-step propagation property
\begin{equation}\label{geo:eq:propagation}
 \overline Q I_t\subseteq I_{t+1},\qquad 1\le t\le L.
\end{equation}
The entire schedule is chosen before drawing the rotation or
running the algorithm.
Any initially supplied physical vector \(Uc_0\) must have its
known canonical vector \(c_0\) in \(I_1\); it is then already
part of the first scheduled disclosure.

Draw \(U\) Haar on \(O(d)\), and define the physical matrix
\[
 Q=U\overline Q U^\top.
\]
Let \(\Xi\) be arbitrary auxiliary physical randomness independent of \(U\).
It is distinct from the static mask tensors \(Z_3,Z_4\).
Its distribution may depend on all public data and on \(d,L\).
It can include all of a policy's private random bits and a finite
array of independent physical Gaussian seeds for reference replies.

Fix a number of response steps \(L_{\rm q}\in\{L-1,L\}\).
For each \(1\le t\le L_{\rm q}\), at the beginning of step \(t\)
disclose \(U|_{I_t}\): the images under \(U\) of a fixed known basis
of \(I_t\). A policy chooses \(x_t\in\R^d\) measurably from \(\Xi\),
all current and past disclosures, and all previous responses. We allow the
\emph{strengthened} response
\begin{equation}\label{geo:eq:strong-response}
 (Qx_t,\ U^\top x_t).
\end{equation}
The second vector is optional for an actual application; giving it
to the policy only makes the geometric theorem stronger.

If \(L_{\rm q}=L-1\), the policy next freezes a terminal candidate
\(x_L\) as a measurable function of its existing history and receives no
response at that point. Only after \(x_L\) has been frozen may the analyst,
for purposes of the geometric proof, disclose \(U|_{I_L}\). The policy has
already terminated and is formally extended to ignore that disclosure.
Thus a terminal candidate contributes one submitted vector but no
matrix-vector product, latent-coordinate response, or oracle observation.

Auxiliary physical responses at the response steps are allowed if they are measurable
functions of \(\Xi\), the queried vectors so far, the responses
\eqref{geo:eq:strong-response} so far, the disclosed restrictions
\(U|_{I_t}\), and earlier auxiliary responses. They may not depend
on undisclosed columns of \(U\) in any other way.
Thus, after fixing \(\Xi\), such responses are deterministic
postprocessing of exactly the matrix-vector and scheduled-column
information just specified.

The horizon \(L\) is a deterministic cap on submitted vectors. A policy
stopping early may be padded with zero response queries while keeping its
terminal candidate frozen. The case \(L_{\rm q}=L-1\) is exactly the model
of an arbitrary unqueried output after \(L-1\) oracle calls.

\subsection{Adaptive Haar completion}
\begin{lemma}[Adaptive Haar completion]\label{geo:lem:completion}
Let \(\mathscr H\) be a standard-Borel history. Suppose that, conditionally
on \(\mathscr H\), the history fixes two subspaces \(E,F\subseteq\R^d\) of
the same dimension and an isometry \(U_E:E\to F\), that \(U|_E=U_E\), and
that the remaining restriction
\[
 U:E^\perp\longrightarrow F^\perp
\]
is Haar on the space of isometries between these orthogonal complements.
Then either of the following one-pair exposures preserves the same property
on the enlarged orthogonal complements.
\begin{enumerate}[label=(\roman*)]
\item If a canonical vector \(v\) is \(\mathscr H\)-measurable and its image
      \(w=Uv\) is revealed, set
      \[
       e=(I-P_E)v,\qquad f=(I-P_F)w=Ue,
       \qquad E^+=E+\Span\{e\},\quad F^+=F+\Span\{f\}.
      \]
\item If a physical vector \(x\) is \(\mathscr H\)-measurable and its
      preimage \(\xi=U^\top x\) is revealed, set
      \[
       f=(I-P_F)x,\qquad e=(I-P_E)\xi=U^\top f,
       \qquad E^+=E+\Span\{e\},\quad F^+=F+\Span\{f\}.
      \]
\end{enumerate}
Conditionally on the enlarged history, \(U:E^{+\perp}\to F^{+\perp}\)
is Haar. This includes the zero-residual cases \(e=f=0\), remains true after
adjoining arbitrary measurable deterministic postprocessing of the exposed
pairs, and may be iterated a finite number of times.
\end{lemma}

\begin{proof}
Fix a value of \(\mathscr H\) for which the asserted conditional law holds.
Choose orthonormal frames for \(E^\perp\) and \(F^\perp\); under these
identifications the unused restriction of \(U\) is a Haar matrix
\(V\in O(n)\), where \(n=d-\dim E\). Measurable frames can be chosen by
Gram--Schmidt in a deterministic coordinate order, so the argument defines
a regular conditional version as the history varies.

For (i), if \(e\ne0\), its normalized coordinate vector is fixed before the
exposure and \(Ve/\|e\|\) is uniform on the unit sphere. Conditional on its
revealed value \(f/\|f\|\), the matrix \(V\) is uniform on the coset of
orthogonal matrices sending \(e/\|e\|\) to \(f/\|f\|\). Choose measurable
orthogonal maps carrying the first coordinate vector to these two unit
vectors. After left and right multiplication by those maps, the unfixed
lower-right \((n-1)\times(n-1)\) block is Haar on \(O(n-1)\), with
\(O(0)\) interpreted as the singleton group. Equivalently, the
restriction between \(E^{+\perp}\) and \(F^{+\perp}\) is Haar. This is also
the unique conditional law invariant under left and right multiplication by
the stabilizers of the exposed pair. Case (ii) is the same argument applied
to \(U^\top\). If the residual is zero, no complement is reduced and the
conditional law is unchanged. A measurable function of already exposed data
does not enlarge the generated sigma-field beyond that data, and finite
iteration proves the last assertion.
\end{proof}

\subsection{Deferred Haar representation}
\begin{lemma}[Canonical Gaussian starts]\label{geo:lem:haar}
Assume \(d>D_0+2L\). On an enlargement of the probability space,
there are canonical vectors \(g_1,\ldots,g_L\in\R^d\) whose joint
law, conditional on \(\Xi\), is that of independent standard Gaussian
vectors, with the following property.
Set
\[
 {\cal R}
 =\cyc_{\overline Q}\bigl(\Span\{g_1,\ldots,g_L\}\bigr)
 =\bigoplus_{\lambda}
       \Span\{P_\lambda g_1,\ldots,P_\lambda g_L\}.
\]
Then simultaneously for \(1\le t\le L\),
\begin{equation}\label{geo:eq:query-containment}
 U^\top x_t\in I_t+{\cal R}.
\end{equation}
Moreover, at every response step \(1\le t\le L_{\rm q}\),
\[
 U^\top Qx_t\in I_{t+1}+{\cal R}.
\]
The vectors \(g_t\) need not be independent of \(U\).
That independence is neither asserted nor used.
\end{lemma}

\begin{proof}
Fix the independent physical advice \(\Xi=z\), using a fixed regular
conditional version supplied by the standard-Borel convention in the oracle
model. We describe the usual deferred-rotation construction explicitly.
All conditional isometries can be expressed using measurable orthonormal
frames obtained by Gram--Schmidt in a deterministic coordinate order,
omitting zero residuals. This fixes measurable versions of the constructions
also when exposure dimensions change.

Just before processing the submitted vector \(x_t\) in the deferred-exposure
construction, let \({\cal L}_t\) be the span of
the canonical vectors whose images under \(U\) have been exposed:
the disclosed subspace \(I_t\), previous
\(\xi_i=U^\top x_i\), and previous \(\overline Q\xi_i\),
for \(i<t\). Their images are known, respectively, from the
scheduled disclosures, the chosen physical queries \(x_i\), and
the matrix responses \(Qx_i\).
The conditioning history will also include the auxiliary Gaussian
variables constructed below, but \({\cal L}_t\) itself consists
only of the subspace on which the image under \(U\) has been
exposed. A Gaussian start corresponding to a zero perpendicular
query component is independent of \(U\) and does not disclose its
own image. In the nonzero case its image is already determined
by the exposed query pair, so it does not enlarge this subspace.
In particular
\[
 \dim{\cal L}_t\le D_0+2(t-1)<d.
\]

Conditional on this history, the remaining restriction
\[
 U:{\cal L}_t^\perp\longrightarrow(U{\cal L}_t)^\perp
\]
is Haar. Indeed, start from the unconditional Haar law, expose each scheduled
canonical basis vector by Lemma~\ref{geo:lem:completion}(i), expose the
preimage of each adaptively chosen physical query by part (ii), and, at a
response step, expose the image of \(\overline Q\xi_i\) by part (i).
All auxiliary responses are deterministic postprocessing of these pairs and
the fixed \(z\), so the last clause of the lemma shows that they impose no
additional constraint. At a terminal candidate step, \(x_L\) was frozen
before the analyst-only disclosure of \(U|_{I_L}\); it remains measurable
after that disclosure, so the same argument applies to its preimage.

Decompose
\[
 x_t=x_t^\parallel+x_t^\perp,\qquad
 x_t^\parallel\in U{\cal L}_t,\quad
 x_t^\perp\perp U{\cal L}_t.
\]
If \(x_t^\perp\ne0\), its latent direction
\[
 e_t=U^\top x_t^\perp/\norm{x_t^\perp}
\]
is conditionally uniform on the unit sphere of \({\cal L}_t^\perp\).
Independently of this direction and of the unused rotation, sample
\[
 a_t\sim N(0,P_{{\cal L}_t}),\qquad
 R_t\sim\chi_{d-\dim{\cal L}_t},
\]
and put \(g_t=a_t+R_te_t\).
Given the preceding augmented history and \(\Xi=z\), this vector
has exactly the \(N(0,I_d)\) distribution. Its projection on
\({\cal L}_t^\perp\) has a Gaussian radius and uniform independent
direction, while its component in \({\cal L}_t\) is an independent
Gaussian. If \(x_t^\perp=0\), instead take an independent
\(g_t\sim N(0,I_d)\).
Consequently the whole sequence \(g_1,\ldots,g_L\) is jointly iid
standard Gaussian conditional on \(\Xi=z\).

In either case,
\begin{equation}\label{geo:eq:new-start}
 \xi_t=U^\top x_t\in{\cal L}_t+\Span\{g_t\}.
\end{equation}
When the perpendicular component is nonzero, revealing \(g_t\)
after the response adds only the auxiliary variables \(a_t,R_t\)
to information already contained in the exposed direction
of \(\xi_t\). Their physical images are determined on the now
exposed subspace. The unused restriction of \(U\) remains Haar.
If \(t\le L_{\rm q}\), the response \(Qx_t\), together with \(\xi_t\),
additionally fixes the image of \(\overline Q\xi_t\). Lemma
\ref{geo:lem:completion}(i) then proves the conditional-Haar induction for
the next step. If \(t=L=L_{\rm q}+1\), no response is produced and no later
induction step is needed.

To prove the containment, it is convenient to use the full
cyclic span \({\cal R}\) of all \(L\) Gaussian starts in the
analysis. At step \(1\), \({\cal L}_1=I_1\), so
\eqref{geo:eq:new-start} implies \(\xi_1\in I_1+{\cal R}\).
Inductively, if \(\xi_i\in I_i+{\cal R}\) for \(i<t\), then
\[
 \overline Q\xi_i\in I_{i+1}+{\cal R}
                  \subseteq I_t+{\cal R}
\]
by \eqref{geo:eq:propagation} and invariance of \({\cal R}\).
The definition of \({\cal L}_t\) therefore implies
\({\cal L}_t\subseteq I_t+{\cal R}\).
Equation \eqref{geo:eq:new-start} gives \eqref{geo:eq:query-containment}.
At a response step, applying \(\overline Q\) gives the response containment.
\end{proof}

\begin{remark}[Why no independence from \(U\) is needed]
The canonical Gaussian starts are generally correlated with \(U\).
For example, the first fixed physical query with no disclosed
subspace makes the first start's direction equal to
\(U^\top x_1/\norm{x_1}\). What matters below is their iid
Gaussian law relative to the \emph{fixed canonical} objects
\(\overline Q,A_0,I_t\), conditional on independent physical
advice. No probabilistic claim is made after conditioning on \(U\).
\end{remark}

\begin{remark}[The strengthened latent-coordinate response]
The optional response \(U^\top x_t\) fixes the latent vector of the
one chosen physical query; together with \(Qx_t\), it exposes
the two pairs \((\xi_t,x_t)\) and
\((\overline Q\xi_t,Qx_t)\). It does not disclose the individual
columns of \(U\) on the support of \(\xi_t\).
Querying an already known \(Uc_0\) returns the already known
canonical \(c_0\); its matrix product provides the image of
\(\overline Qc_0\), which is precisely the allowed one-step
propagation. In particular, the argument never grants spectral
projectors applied to \(c_0\) or to any disclosed \(I_t\).
\end{remark}

\subsection{A uniform random-subspace angle bound}
\begin{lemma}[High multiplicity]\label{geo:lem:angle}
Let \(0<\zeta<1\), and define
\[
 s_0=(L+1)\log17+\log(4J/\zeta).
\]
If
\[
 K+1\ge128s_0,
\]
then, conditional on every value of the independent physical
advice, with probability at least \(1-\zeta\),
\begin{equation}\label{geo:eq:angle}
 \norm{P_{\cal R}|_{A_0}}\le\eta,\qquad
 \eta^2=\frac{32s_0}{K+1}\le\frac14.
\end{equation}
Here the probability includes the Haar rotation and the auxiliary
Gaussian construction of Lemma \ref{geo:lem:haar}.
\end{lemma}

\begin{proof}
In a fixed eigenspace of multiplicity
\(D_\lambda=(K+1)m_\lambda\), the matrix
\[
 G_\lambda=
 [P_\lambda g_1\ \cdots\ P_\lambda g_L]
\]
has iid standard Gaussian entries in an orthonormal coordinate
system. Its column span is \({\cal R}_\lambda\).
The active subspace in this eigenspace,
\(A_{0,\lambda}=A_0\cap\range(P_\lambda)\), has dimension
\(m_\lambda\).

We record explicit elementary estimates. If \(G\) is a
\(D\times L\) iid Gaussian matrix, a \(1/8\)-net of the unit
sphere of \(\R^L\) can be chosen with at most \(17^L\) points.
For a fixed unit vector \(v\), \(\norm{Gv}^2\) is chi-square
with \(D\) degrees of freedom. Chernoff bounds give
\[
 \Pr\{\norm{Gv}^2\notin[D/2,2D]\}\le2e^{-D/16}.
\]
For example, the lower and upper tails follow by applying the
moment-generating function at parameters \(-1/2\) and \(1/4\);
both resulting exponential constants exceed \(1/16\).
If
\[
 D\ge16[L\log17+\log(4J/\zeta)],
\]
a union bound over the net therefore fails with probability at
most \(\zeta/(2J)\).
On the resulting event,
\[
 \norm G\le\frac87\sqrt{2D},\qquad
 s_{\min}(G)\ge
 \sqrt{D/2}-\frac18\frac87\sqrt{2D}
 >\frac12\sqrt D.
\]

Next let \(A\) be any fixed \(m\)-dimensional subspace.
The matrix \(P_A G\), expressed in an orthonormal basis of \(A\),
is \(m\times L\) iid Gaussian. Take \(1/4\)-nets in its two
unit spheres, of cardinalities at most \(9^m\) and \(9^L\).
Every fixed bilinear form on the two nets is a standard Gaussian.
With failure probability at most \(\zeta/(2J)\),
\[
 \norm{P_A G}\le
 2\sqrt{2[(m+L)\log9+\log(4J/\zeta)]}.
\]
The factor two comes from approximating both test vectors:
the two net errors sum to at most half of the operator norm.

Apply these estimates with
\(D=(K+1)m_\lambda\), \(m=m_\lambda\).
The assumed lower bound on \(K+1\) implies the displayed
condition on \(D\). On their joint event,
\[
 \norm{P_{{\cal R}_\lambda}|_{A_{0,\lambda}}}
 \le \frac{\norm{P_{A_{0,\lambda}}G_\lambda}}
              {s_{\min}(G_\lambda)}.
\]
Squaring and substituting the bounds gives
\[
 \norm{P_{{\cal R}_\lambda}|_{A_{0,\lambda}}}^2
 \le
 \frac{32[(m_\lambda+L)\log9+\log(4J/\zeta)]}
      {(K+1)m_\lambda}
 \le\frac{32s_0}{K+1}.
\]
A union bound over the \(J\) eigenspaces costs at most \(\zeta\).
Since both \(A_0\) and \({\cal R}\) are orthogonal sums over the
same eigenspaces, their projection operator norm is the maximum
of these block norms. This proves \eqref{geo:eq:angle}.
Independence of the eigenspace events is not needed.
\end{proof}

\subsection{Hidden readouts and the whole query span}
\begin{theorem}[Scheduled hidden-direction bound]\label{geo:thm:hidden}
Assume the model above, \(d>D_0+2L\), and
\(K+1\ge128[(L+1)\log17+\log(4J/\zeta)]\).
With probability at least \(1-\zeta\), simultaneously for every
\(1\le t\le L\) and every \(a\in A_0\cap I_t^\perp\),
\begin{equation}\label{geo:eq:hidden}
 \norm{P_{\Span\{x_1,\ldots,x_t\}}Ua}
 \le\frac{\eta}{\sqrt{1-\eta^2}}\norm a
 \le2\eta\norm a,
\end{equation}
where
\[
 \eta^2=\frac{32[(L+1)\log17+\log(4J/\zeta)]}{K+1}.
\]
The conclusion is uniform over arbitrary measurable adaptive
policies and allowed auxiliary reference responses.
\end{theorem}

\begin{proof}
Work on the event of Lemma \ref{geo:lem:angle}.
For \(i\le t\), Lemma \ref{geo:lem:haar} and nestedness give
\[
 U^\top x_i\in I_i+{\cal R}\subseteq I_t+{\cal R}.
\]
It suffices to bound the projection of \(a\) onto
\(I_t+{\cal R}\). For \(v\in{\cal R}\),
\[
 \norm{P_{I_t}v}\le\norm{P_{A_0}v}\le\eta\norm v,
\]
because projection operator norms are unchanged by transposition.
Thus for \(w=i+v\), \(i\in I_t\),
\[
 \norm w\ge\norm{P_{I_t^\perp}v}
          \ge\sqrt{1-\eta^2}\norm v.
\]
Since \(a\perp I_t\), and \(a\in A_0\),
\[
 |\langle a,w\rangle|
   =|\langle a,v\rangle|
   \le\eta\norm a\norm v
   \le\frac{\eta}{\sqrt{1-\eta^2}}\norm a\norm w.
\]
Taking the supremum over unit vectors in \(I_t+{\cal R}\)
proves \eqref{geo:eq:hidden}. The constants give \(\eta\le1/2\),
so the last inequality follows.
\end{proof}

\begin{corollary}[Readout calendar]\label{geo:cor:calendar}
Let fixed canonical readouts \(a_s\in A_0\) have deterministic
disclosure times \(\tau_s\), with
\[
 a_s\perp I_t\quad(t<\tau_s),\qquad
 a_s\in I_t\quad(t\ge\tau_s).
\]
Then \eqref{geo:eq:hidden} applies simultaneously to every \(a_s\)
at every time \(t<\tau_s\), without an additional union bound over
the readouts or times. If \(\norm{x_i}\le R\), it implies
\[
 |\langle Ua_s,x_i\rangle|\le2\eta R\norm{a_s},
 \qquad i\le t<\tau_s.
\]
If a query is replaced inside the function by a radial squash
\(\rho(x)=q(\norm x)x\) with \(\norm{\rho(x)}\le R\), the same
bound holds for \(\rho(x_i)\), even for unbounded physical
queries, because radial squashing preserves their span.
\end{corollary}

\paragraph{Example: the analytic block-chain layer schedule.}
Suppose the canonical active coordinates are indexed by
\((s,j)\), \(1\le s\le N\), \(1\le j\le M\), and the readout in
block \(s\) is supported in the suffix \(j\ge j_\star\).
Assign layer
\[
 \ell(s,j)=(s-1)j_\star+\min\{j,j_\star\}.
\]
Let \(I_t\) be the span of active coordinates of layer at most \(t\).
Every nonzero quadratic entry joins layers differing by at most one:
within-block path edges join consecutive layers, suffix stabilization joins
equal layers, and interblock edges join a suffix of layer $sj_\star$ to the
next entry of layer $sj_\star+1$. Diagonal entries preserve layers. These
exhaust the matrix support and prove
\(\overline QI_t\subseteq I_{t+1}\).
Every coordinate of readout \(a_s\) has layer \(sj_\star\), so one
can choose \(\tau_s=sj_\star\). The simultaneous suffix-coordinate
jump in \(I_t\) is allowed by the theorem; its dimension may jump
by more than one at a disclosure time.

\subsection{Reference-noise conditioning and what is not granted}
\begin{corollary}[Innovation Gaussian reference responses]
Suppose a reference response uses exogenous physical Gaussian
vectors \(\zeta_{k,t}\), independent of \(U\), and has the form
\[
 Y_{k,t}=\mu_{k,t}
       +\Sigma_{k,t}^{1/2}\zeta_{k,t},
\]
where \(\mu_{k,t}\) and \(\Sigma_{k,t}\) are measurable functions
of physical query history, earlier responses, matrix products,
and planted tensors formed only from currently disclosed
readouts \(Ua_s\in UI_t\). Then Theorem \ref{geo:thm:hidden} applies
to the adaptive query policy driven by these responses.
\end{corollary}

\begin{proof}
Include the finite collection of \(\zeta_{k,t}\) in the physical
advice \(\Xi\). This advice is independent of \(U\).
Conditional on it, every reference reply is deterministic
postprocessing of the allowed transcript and scheduled images.
It therefore imposes no additional constraint on unused columns
of \(U\) beyond the pairs tracked in Lemma \ref{geo:lem:haar}.
The theorem is conditional on each value of this advice and hence
also holds after averaging over it.
The mean may include a least-norm fit of earlier tensor responses,
and the covariance may include orthogonal projections determined
by all past physical query vectors. Those calculations are
functions of the permitted history and cause no exception.
\end{proof}

\paragraph{Limits of this reduction.}
An auxiliary mean depending on an undisclosed planted readout is
not covered. Nor is an auxiliary covariance depending directly
on an undisclosed orientation of \(U\).
The theorem does not condition on the entire physical matrix \(Q\);
the policy receives only the counted matrix-vector products.
It also never discloses \(P_\lambda U^\top c\) or the images of the
spectral closure of \(I_t\).
The spectral closure is taken only for the canonical Gaussian
innovations in an upper bound on the information span.

Auxiliary gradient-noise vectors themselves need not lie in
\(U(I_t+{\cal R})\) at the instant they are returned. This is
unnecessary: once a policy uses any function of those vectors as
its next query or its final output, it is covered by the
one-new-query Haar argument. The theorem bounds the entire span
of query arguments, which is the subspace needed in tensor
observation projection bounds.

\paragraph{Uniformity without a common event for all policies.}
The deferred-Haar coupling and its geometric good event can depend on the
policy. Their probability estimates are uniform in the policy, and the
canonical matrix and the distribution of the original rotation do not change.
A single good event holding simultaneously for every policy is neither
asserted nor needed. The Gaussian starts are auxiliary variables in a coupling,
not part of the hard-instance prior.

\section{Exact Gaussian comparison and transfer to fixed functions}\label{app:gaussian}
This appendix proves the information comparison independently of the analytic
construction. The Gaussian calculations use the unconditioned static-tensor
prior. The final proposition alone conditions on the norm event required for
admissibility.

\subsection{Observation model and exact adaptive Gaussian comparison}
Fix positive integers $d,N,T$ and let
$\mathcal H_k=\Sym^k(\R^d)$, $k\in\{3,4\}$, with the Frobenius inner product
inherited from $(\R^d)^{\otimes k}$.  Symmetrization is the average over all
permutations.  In particular $\|b^{\otimes k}\|_F=1$ for a unit vector $b$.
A standard Gaussian $Z_k$ in $\mathcal H_k$ means an isonormal Gaussian for
this inner product.  Let $Z_3,Z_4$ be independent.

Let $U$ be an auxiliary random parameter, independent of $Z_3,Z_4$ and of the
algorithm's random seed $\omega$.  In the application $U$ is a Haar orthogonal
matrix.  Let $b_s=b_s(U)\in\R^d$, $1\leq s\leq N$, be unit vectors, and put
\[
 S_k(U)=w_k\sum_{s=1}^N b_s^{\otimes k},\qquad
 \Theta_k=S_k(U)+\sigma_k Z_k,\qquad \sigma_k>0.
\]
The $b_s$ may be orthogonal, but orthogonality is not needed for the Gaussian
comparison below.  Fix disclosure times
$\tau_s\in\{1,\ldots,T+1\}\cup\{\infty\}$ and define
\[
 S_{k,\vis}(t)=w_k\sum_{s:\,\tau_s\leq t}b_s^{\otimes k},\qquad
 S_{k,\hid}(t)=w_k\sum_{s:\,t<\tau_s}b_s^{\otimes k}.
\]

Before query $t$, an algorithm receives specified side information about $U$,
including the vectors $b_s$ with $\tau_s\leq t$.  It selects a vector $z_t$
as an arbitrary measurable function of its seed and its available history.
Any deterministic radial reparametrization of an unrestricted query may be
included in this policy.  It then receives the two tensor channels
\begin{equation}\label{gau:eq:actual-channel}
 Y_{k,t}=L_{k,t}\Theta_k,
 \qquad L_{k,t}A=A[z_t^{\otimes(k-1)},\,\cdot]\in\R^d.
\end{equation}
The two answers are revealed in the fixed order $k=3$ and then $k=4$, with no
policy action between the two reveals. Lemma~\ref{gau:lem:factorization}
below proves the product-kernel property for the true static tensors; for the
reference it follows immediately from the independent current innovation
seeds. Hence the two conditional KL contributions add within the round.
Additional side answers are allowed if, conditional on $U$, they are
deterministic functions of the selected queries and the existing history.
For example, a matrix-vector product $Q_Uz_t$, scheduled columns of $U$, and
corrections computable from disclosed readouts satisfy this condition.
These side answers must not reveal additional information about the
unobserved Gaussian tensors.

For each realized history define
\begin{align*}
 V_{k,t}
 &=\spanop\{\Sym(z_i^{\otimes(k-1)}\otimes v):
             1\leq i\leq t,\ v\in\R^d\},\qquad V_{k,0}=\{0\},\\
 P_{k,t}&=P_{V_{k,t}},\qquad
 \Delta V_{k,t}=V_{k,t}\cap V_{k,t-1}^{\perp},\\
 B_{k,t}&=L_{k,t}(I-P_{k,t-1}),\qquad
 C_{k,t}=B_{k,t}B_{k,t}^{*}
        =L_{k,t}(I-P_{k,t-1})L_{k,t}^{*}.
\end{align*}
Let $M_{k,t-1}\in V_{k,t-1}$ be the minimum Frobenius norm tensor fitting all
past answers in channel $k$.  Set $M_{k,0}=0$.  All these objects are computed
from the process's own query and answer history.

\begin{lemma}[Adaptive two-channel posterior factorization]
\label{gau:lem:factorization}
Let $\mathcal F_{t-1}$ denote the complete pre-query history, including all
allowed side information but excluding the current query and current tensor
answers. Under the true static-tensor experiment, for every $t$ and for a
fixed regular conditional version,
\begin{align}
&\mathcal L\bigl(\Theta_3,\Theta_4
       \mid \mathcal F_{t-1},U,\omega\bigr)
\nonumber\\
&\qquad=
\bigotimes_{k=3}^4
N_{\mathcal H_k}\!\left(
 M_{k,t-1}+(I-P_{k,t-1})S_k,
 \sigma_k^2(I-P_{k,t-1})
\right),
\label{gau:eq:posterior-product}
\end{align}
where the displayed Gaussians may be degenerate on their affine supports.
Consequently, conditional on the pre-round history and the predictable common
query $z_t$, $(Y_{3,t},Y_{4,t})$ has the product of the two channel kernels.
\end{lemma}

\begin{proof}
Condition on $U$ and $\omega$. At time zero the claim follows from the
independence of $Z_3$ and $Z_4$. Suppose it holds before round $t$. The query
$z_t$, and hence each linear map $L_{k,t}$, is measurable with respect to the
pre-round history. Conditional on that history, applying $L_{3,t}$ and
$L_{4,t}$ to the two independent posterior factors therefore produces a
product response kernel. Conditioning on the two realized answers preserves
the product form, because the joint likelihood is the product of one
likelihood depending only on $\Theta_3$ and one depending only on $\Theta_4$.
Allowed side answers are deterministic once $U$ and the existing history are
fixed and hence add no tensor likelihood factor.

For a single channel, the combined past observation map has row space
$V_{k,t-1}$ and kernel $V_{k,t-1}^{\perp}$. Its affine fiber therefore
determines precisely the projection of $\Theta_k$ onto $V_{k,t-1}$, namely
the unique minimum-norm fitter $M_{k,t-1}$. Orthogonal Gaussian decomposition
then gives
\[
 \Theta_k
 =M_{k,t-1}+(I-P_{k,t-1})S_k
   +\sigma_k(I-P_{k,t-1})\widetilde Z_k,
\]
with $\widetilde Z_k$ standard Gaussian in $\mathcal H_k$. This is the factor
appearing in \eqref{gau:eq:posterior-product}. The induction proves the true
posterior factorization for all rounds. Existence of the regular conditional
versions is guaranteed by the standard-Borel convention in the oracle model.
\end{proof}

\begin{theorem}[Innovation-censored Gaussian reference]\label{gau:thm:innovation}
There is a causal reference process with the same prior for $(U,\omega)$,
the same policy and side-answer rules, and channel answers
\begin{equation}\label{gau:eq:reference}
 Y^0_{k,t}
 =L_{k,t}M_{k,t-1}
  +B_{k,t}S_{k,\vis}(t)
  +\sigma_k C_{k,t}^{1/2}\zeta_{k,t},
\end{equation}
where all physical vectors $\zeta_{k,t}\sim N(0,I_d)$ are independent of each
other and of $(U,\omega)$.  Every reference history admits a tensor fitting
all its past channel answers.

Write $\Ptrue$ for the law of $(U,\omega,\text{transcript})$ generated by
\eqref{gau:eq:actual-channel} and $\Pref$ for its reference law.  Then
\begin{equation}\label{gau:eq:exact-kl}
 \boxed{\quad
 \KL(\Pref\Vert\Ptrue)
 =\sum_{k=3}^4\frac{1}{2\sigma_k^2}
   \sum_{t=1}^T\E_{\Pref}
    \bigl\|P_{\Delta V_{k,t}}S_{k,\hid}(t)\bigr\|_F^2 .
 \quad}
\end{equation}
The formula includes degenerate measurements and repeated queries.  In
particular, repeating a fully determined measurement produces its old
answer in the reference process.
\end{theorem}

\begin{proof}
First condition on $U$, the algorithm's seed, and the full history before
query $t$. By Lemma~\ref{gau:lem:factorization}, the two tensor posteriors
factorize and each channel has the posterior
\begin{equation}\label{gau:eq:posterior}
 \Theta_k\mid\text{history},U,\omega
 \ \sim\
 M_{k,t-1}+(I-P_{k,t-1})S_k
       +\sigma_k(I-P_{k,t-1})\widetilde Z_k,
\end{equation}
where the remaining $\widetilde Z_k$ are conditionally independent standard
Gaussians. Together with the fixed within-round reveal order specified above,
this identifies the joint same-round conditional law as the product of the
two displayed kernels.

Consequently the true conditional law of $Y_{k,t}$ is
\begin{equation}\label{gau:eq:true-kernel}
 N\bigl(L_{k,t}M_{k,t-1}+B_{k,t}S_k,
        \sigma_k^2 C_{k,t}\bigr).
\end{equation}
The reference kernel \eqref{gau:eq:reference} replaces only the new signal
innovation by its currently visible part.  Let
$r_{k,t}=Y^0_{k,t}-L_{k,t}M_{k,t-1}$.  Almost surely
$r_{k,t}\in\range(B_{k,t})=\range(C_{k,t})$, so
\begin{equation}\label{gau:eq:fitter-update}
 M_{k,t}=M_{k,t-1}
   +B_{k,t}^{*}C_{k,t}^{\dagger}r_{k,t}
\end{equation}
fits the old answers and the new one.  It lies in $V_{k,t}$ and is therefore
the unique minimum-norm fitter.  This proves that the reference is defined
inductively at every step.  If $B_{k,t}=0$, both the new mean contribution and
the new noise vanish; this proves consistency for repeated measurements.

At the same realized history the true and reference kernels have the same
covariance.  Their mean difference is $B_{k,t}S_{k,\hid}(t)$, which belongs
to its range.  The KL divergence between these possibly singular Gaussian
kernels is therefore
\begin{align*}
 \frac{1}{2\sigma_k^2}
 \bigl\langle B_{k,t}S_{k,\hid}(t),
 C_{k,t}^{\dagger}B_{k,t}S_{k,\hid}(t)\bigr\rangle
 &=\frac{1}{2\sigma_k^2}
 \bigl\|P_{\range(B_{k,t}^{*})}S_{k,\hid}(t)\bigr\|_F^2\\
 &=\frac{1}{2\sigma_k^2}
 \bigl\|P_{\Delta V_{k,t}}S_{k,\hid}(t)\bigr\|_F^2.
\end{align*}
Here $B^*(BB^*)^{\dagger}B$ is the orthogonal projector onto
$\range(B^*)$. Moreover,
\[
 V_{k,t}=V_{k,t-1}\oplus
 \range\bigl((I-P_{k,t-1})L_{k,t}^*\bigr),
\]
so $\range(B_{k,t}^{*})=\Delta V_{k,t}$.
To justify using the true conditional kernel along reference histories,
we establish absolute continuity by induction. The history laws begin with
the same prior. At each common feasible history the next channel kernels
have the same affine support
\[
 L_{k,t}M_{k,t-1}+\range(B_{k,t}).
\]
Their covariance is positive definite on its range, so both densities on
this affine support are strictly positive. In zero dimension they are the
same point mass. Thus the next reference kernel is absolutely continuous
with respect to the true one. Identical predictable query choices and
deterministic side-answer rules preserve this property. This induction makes
the posterior version valid on reference histories and permits the KL chain
rule. The common prior, predictable queries, and identical deterministic
side answers contribute zero KL. The KL chain rule is applied in the fixed
within-round order $k=3,4$, so conditional independence makes the two
contributions additive. Summing the conditional divergences proves
\eqref{gau:eq:exact-kl}. All spaces are finite dimensional; pseudoinverses and
positive semidefinite square roots give Borel measurable kernels. The regular
conditional versions invoked here are available under the standard-Borel
convention in the oracle model.
\end{proof}

\begin{remark}[Why only the innovation is censored]
Subtracting a hidden signal afresh from a static noisy answer at each time
would change old answers when a signal becomes visible.  That construction
can have singular replay distributions.  In \eqref{gau:eq:reference} every old
measurement component is retained in $M_{k,t-1}$; only directions newly
measured at time $t$ receive the censored mean.  This is the reason the
conditional covariance and support agree exactly with the static model.
\end{remark}

\subsection{Converting small hidden angles to a KL bound}
\begin{lemma}[Rank-one tensor projection]\label{gau:lem:rankone}
Let $W\subseteq\R^d$ be a linear subspace, and suppose
\[
 V\subseteq\spanop\{\Sym(x^{\otimes(k-1)}\otimes v):
                       x\in W,\ v\in\R^d\}.
\]
For every unit vector $b$,
\begin{equation}\label{gau:eq:rankone-bound}
 \|P_V b^{\otimes k}\|_F^2
 \leq k\|P_Wb\|^{2(k-1)}.
\end{equation}
\end{lemma}

\begin{proof}
In the full tensor product space let $E_j$ be the commuting orthogonal
projector having $I$ in slot $j$ and $P_W$ in every other slot.  The space
$V$ is contained in $\sum_{j=1}^k\range(E_j)$, whose orthogonal projector
is bounded above, in positive semidefinite order, by $\sum_jE_j$.
Applying this inequality to $b^{\otimes k}$ gives
\[
 \|P_Vb^{\otimes k}\|_F^2
 \leq\sum_{j=1}^k
 \langle b^{\otimes k},E_jb^{\otimes k}\rangle
 =k\|P_Wb\|^{2(k-1)}.
\]
Restriction to the symmetric subspace does not change the conclusion.
\end{proof}

\begin{remark}[Exact projection formula]
For the full space
$\mathcal V_k(W)=\Span\{\Sym(x^{\otimes(k-1)}\otimes v):x\in W,v\in\R^d\}$,
polarization shows that it consists of symmetric tensor components with at
most one factor in $W^\perp$. Writing $a=\|P_Wb\|$ and decomposing orthogonally
by the number of $W^\perp$ factors gives
\[
 \|P_{\mathcal V_k(W)}b^{\otimes k}\|_F^2
 =a^{2k}+k a^{2(k-1)}(1-a^2)\le k a^{2(k-1)}.
\]
The same upper bound therefore holds for every smaller observation space.
\end{remark}

For a disclosure time $\tau_s$, put
$\ell_s=\min\{T,\tau_s-1\}$, with $\ell_s=T$ if $\tau_s=\infty$, and
\[
 W_s=\spanop\{z_1,\ldots,z_{\ell_s}\}.
\]

\begin{corollary}[Cumulative bound]\label{gau:cor:kl-bound}
The identity \eqref{gau:eq:exact-kl} implies
\begin{equation}\label{gau:eq:kl-moment}
 \KL(\Pref\Vert\Ptrue)
 \leq\frac{N}{2}\sum_{k=3}^4\frac{k w_k^2}{\sigma_k^2}
    \sum_{s=1}^N\E_{\Pref}\|P_{W_s}b_s\|^{2(k-1)}.
\end{equation}
If an event of reference probability at least $1-\zeta$ guarantees
$\|P_{W_s}b_s\|\leq h$ simultaneously for all $s$, then
\begin{equation}\label{gau:eq:kl-angle}
 \KL(\Pref\Vert\Ptrue)
 \leq\frac{N^2}{2}\sum_{k=3}^4
 \frac{k w_k^2}{\sigma_k^2}
       \bigl(h^{2(k-1)}+\zeta\bigr).
\end{equation}
\end{corollary}

\begin{proof}
At every realized reference path, Cauchy--Schwarz yields
\begin{align*}
 \sum_{t=1}^T\|P_{\Delta V_{k,t}}S_{k,\hid}(t)\|_F^2
 &\leq Nw_k^2\sum_{s=1}^N
       \sum_{t<\tau_s,\ t\leq T}
          \|P_{\Delta V_{k,t}}b_s^{\otimes k}\|_F^2\\
 &=Nw_k^2\sum_{s=1}^N
       \|P_{V_{k,\ell_s}}b_s^{\otimes k}\|_F^2\\
 &\leq kNw_k^2\sum_{s=1}^N\|P_{W_s}b_s\|^{2(k-1)}.
\end{align*}
The middle equality uses orthogonality of the increments of the nested
measurement spaces; it remains a pathwise identity for adaptive queries.
The last inequality is Lemma~\ref{gau:lem:rankone}.  Taking expectations proves
\eqref{gau:eq:kl-moment}.  Since a projection of a unit vector has norm at most
one, the stated event gives each moment at most
$h^{2(k-1)}+\zeta$, proving \eqref{gau:eq:kl-angle}.
\end{proof}

\subsection{Exact quartic cores, first-exit coupling, and fixed functions}
Suppose a smooth scalar primitive $\phi_r$ has the exact core
\begin{equation}\label{gau:eq:quartic-core}
 \phi_r(u)=q_0+q_3u^3+q_4u^4\quad\text{for }|u|\leq r.
\end{equation}
For the tapered primitive in the hard chain,
$(q_0,q_3,q_4)=(10,-40,30)$.
Consider the fixed base function
\begin{equation}\label{gau:eq:fixed-base}
 F_U(z)=C+c_U^Tz+\tfrac12z^TQ_Uz
              +\tau\sum_{s=1}^N\phi_r(A b_s^Tz).
\end{equation}
Here $A,\tau>0$ and $Q_U,c_U$ depend only on $U$.
Set
\[
 w_3=3\tau q_3A^3,\qquad w_4=4\tau q_4A^4,
\]
and define the correction
\[
 D_s(z)=\tau\bigl[
 \phi_r(A b_s^Tz)-q_0-q_3(A b_s^Tz)^3-q_4(A b_s^Tz)^4\bigr].
\]
Both $D_s(z)$ and its gradient vanish whenever $|A b_s^Tz|\leq r$.
The static, globally defined hard function is
\begin{equation}\label{gau:eq:fixed-hard-function}
 G_{U,Z}(x)=(F_U+P)(\rho_R(x))+\frac{\eta_0}{2}\|x\|^2,
 \qquad \eta_0\geq0.
\end{equation}

To use the channel model, at query $x_t$ set $z_t=\rho_R(x_t)$ and combine
the true tensor channels into the following first-order response:
\begin{align}
 \widetilde f_t
 &=C+N\tau q_0+c_U^Tz_t+\tfrac12z_t^TQ_Uz_t
    +\sum_{k=3}^4\frac1k z_t^TY_{k,t}
    +\sum_{s:\tau_s\leq t}D_s(z_t)
    +\tfrac{\eta_0}{2}\|x_t\|^2,\label{gau:eq:assembled-value}\\
 \widetilde g_t
 &=D\rho_R(x_t)^T
   \left(c_U+Q_Uz_t+Y_{3,t}+Y_{4,t}
        +\sum_{s:\tau_s\leq t}\nabla D_s(z_t)\right)
      +\eta_0x_t.\label{gau:eq:assembled-grad}
\end{align}
The visible corrections use only already disclosed $b_s$.  The identities
for function values follow from Euler's identity for each homogeneous
polynomial; values do not add tensor measurements beyond the two channels.

Call a query a \emph{hidden core exit} if
$|A b_s^Tz_t|>r$ for some $s$ with $t<\tau_s$.
The response \eqref{gau:eq:assembled-value}--\eqref{gau:eq:assembled-grad} agrees
exactly with $(G_{U,Z}(x_t),\nabla G_{U,Z}(x_t))$ at every query before,
and also at every query not constituting, a hidden core exit.  In particular,
using the same $(U,Z,\omega)$ to run the assembled-channel algorithm and the
actual fixed-function algorithm couples their transcripts until their first
hidden core exit.  The first exit occurs at the same query in both runs.

\begin{proposition}[Transfer to the fixed-instance distribution]
\label{gau:prop:transfer}
Assume all of the following.
\begin{enumerate}
\item On $\mathcal G_Z$, every $G_{U,Z}$ belongs to the required function
class.  There is a designated last readout $b_N$ with $\tau_N>T$ and a
threshold $h_{\mathrm{cert}}>0$ such that
\begin{equation}\label{gau:eq:certificate-hyp}
 |b_N^T\rho_R(x)|\leq h_{\mathrm{cert}}
 \quad\Longrightarrow\quad
 \|\nabla G_{U,Z}(x)\|>\varepsilon
 \qquad\text{on }\mathcal G_Z.
\end{equation}
This hypothesis may combine an inner-region hard-chain certificate with an
outer-region radial guard certificate.
\item Under the innovation reference, an event $\mathcal G_{\mathrm{geo}}$
of probability at least $1-\zeta$ guarantees both
$\|P_{W_s}b_s\|\leq h$ for every $s$ and
$|b_N^Tz_{\mathrm{out}}|\leq Rh$, where
$z_{\mathrm{out}}=\rho_R(x_{\mathrm{out}})$ and the terminal output is
allowed to be unqueried.
\item The parameters satisfy
$ARh\leq r$ and $Rh\leq h_{\mathrm{cert}}$.
\end{enumerate}
Let $\mathfrak K$ be the right-hand side of \eqref{gau:eq:kl-angle}.  For any measurable
algorithm making at most $T$ queries, its success probability averaged over
the static distribution $(U,Z)\mid\mathcal G_Z$ and its independent seed
satisfies
\begin{equation}\label{gau:eq:success-transfer}
 \boxed{\quad
 \Pr\{\|\nabla G_{U,Z}(x_{\mathrm{out}})\|\leq\varepsilon
          \mid\mathcal G_Z\}
 \leq\frac{\zeta+\sqrt{\mathfrak K/2}}{1-4e^{-4d}}.
 \quad}
\end{equation}
In particular, if the displayed upper bound is below $2/3$, there is one
fixed admissible function on which the algorithm's success probability over
its private seed is below $2/3$.
\end{proposition}

\begin{proof}
Under the assembled true-channel law $\Ptrue$, let $\mathcal B$ be the event
of at least one hidden core exit or
$|b_N^Tz_{\mathrm{out}}|>h_{\mathrm{cert}}$.
It is an event of $(U,\omega,\text{transcript})$; the full Gaussian tensors
need not be included in that joint law.
The geometric event excludes $\mathcal B$ under the reference, since
$\|z_t\|\leq R$ and
$|b_s^Tz_t|\leq R\|P_{W_s}b_s\|$ before disclosure.
Pinsker's inequality and Theorem~\ref{gau:thm:innovation} give
\[
 \Ptrue(\mathcal B)
 \leq\Pref(\mathcal B)+\TV(\Pref,\Ptrue)
 \leq\zeta+\sqrt{\mathfrak K/2}.
\]
Couple the actual fixed-function run with the assembled true-channel run
using their same $(U,Z,\omega)$.  If $\mathcal B$ does not occur, there is no
hidden core exit, so their full transcripts and terminal outputs agree.
On $\mathcal G_Z$, hypothesis \eqref{gau:eq:certificate-hyp} then rules out
success.  Pointwise, actual success together with $\mathcal G_Z$ is thus
contained in the event $\mathcal B$ of the coupled assembled run.  Dividing
the resulting probability bound by
$\Pr(\mathcal G_Z)\geq1-4e^{-4d}$ proves
\eqref{gau:eq:success-transfer}.

The conditioning event depends only on the static instance tensors and is
independent of the algorithm's private seed.  The hard function is sampled
once, before that seed is used.  Averaging the conditional success
probabilities over these fixed functions proves the last assertion.
\end{proof}

\begin{remark}[Conditioning and query conventions]
The Gaussian posterior and KL calculation are performed before conditioning
on $\mathcal G_Z$.  Conditioning the actual distribution is the last step
of Proposition~\ref{gau:prop:transfer}; one never treats a norm-truncated tensor
as Gaussian.  Nor does one condition the KL comparison on the full static
tensors $Z_3,Z_4$.

An algorithm that stops early can be padded with zero queries, keeping its
output frozen.  The geometric lemma includes the final output as one extra
response-free candidate by taking $L_{\rm q}=L-1$. The candidate is frozen
before the analyst-only disclosure at its geometric step and receives no
matrix, latent-coordinate, tensor, value, or gradient response. The KL sum
therefore includes only actual oracle observations. If the model imposes a
per-instance almost-sure query cap,
first truncate the algorithm after $T$ queries on all answer histories.
On each admissible fixed function this changes behavior on a null event
only; integration over the admissible instance distribution preserves that
fact.  This avoids imposing a uniform exceptional set on all private seeds.
\end{remark}

\section{Explicit parameters and the full fixed-distribution proof}\label{app:assembly}
We record the exponent arithmetic and then specify the complete strengthened
experiment, so that the reduction to the original oracle can be checked
without an implicit change of quantifiers.

\subsection{Repeated quadratic copies and a polynomial dimension choice}
We apply the geometric theorem to the Gaussian comparison and record an explicit polynomial dimension.
Let $Q_0$ act on a canonical active space $\mathcal A_0$ of dimension $D_0$,
and put
\[
 \overline Q=I_{K+1}\otimes Q_0,\qquad
 d=(K+1)D_0,\qquad Q_U=U\overline Q U^T.
\]
All linear and nonlinear planted terms lie in the first copy.  The other
copies contain only the positive semidefinite quadratic $Q_0$.
Let $I_t\subseteq\mathcal A_0$ be scheduled canonical disclosure spaces
with $\overline Q I_t\subseteq I_{t+1}$, and disclose their images under
$U$.  The disclosure schedule for $b_s$ must ensure that its canonical
vector is orthogonal to the relevant earlier $I_t$.

The required geometric lemma is the following quantitative implication:
after at most $L$ adaptive query/candidate vectors, with $L_{\rm q}=L-1$
when the last vector is a terminal response-free output, and with
$J\leq D_0$ distinct eigenvalues,
one may take
\begin{equation}\label{gau:eq:geometry-input}
 \eta^2=
 \frac{32\bigl[(L+1)\log17+\log(4J/\zeta)\bigr]}{K+1},
 \qquad h=2\eta,
\end{equation}
provided $\eta\leq1/2$, with failure probability at most $\zeta$.
The geometric statement is applied after fixing the independent physical
reference innovation seeds and the algorithm seed.  Then the reference
is a deterministic adaptive matrix-vector/disclosure policy.  The deferred
canonical Gaussian starts need only have the asserted iid law conditional
on these fixed seeds; they are not asserted to be independent of $U$.

For clarity, if the canonical random cyclic space $\mathcal R$ satisfies
$\|P_{\mathcal R}|_{\mathcal A_0}\|\leq\eta$, then a unit
$b\in\mathcal A_0\cap I_t^\perp$ obeys
\[
 \|P_{I_t+\mathcal R}b\|
 \leq\frac{\eta}{\sqrt{1-\eta^2}}\leq2\eta.
\]
Indeed $I_t+\mathcal R=I_t\oplus(I-P_{I_t})\mathcal R$; for
$r\in\mathcal R$, the numerator $|\langle b,r\rangle|$ is at most
$\eta\|r\|$, and
$\|(I-P_{I_t})r\|\geq\sqrt{1-\eta^2}\|r\|$.

Combining \eqref{gau:eq:gaussian-constant}, \eqref{gau:eq:kl-angle}, and
\eqref{gau:eq:geometry-input} gives, for an absolute constant $C$,
\begin{equation}\label{gau:eq:general-calibration}
 \mathfrak K\leq C dN^2
 \left[
 \frac{w_3^2}{\delta_3^2}(h^4+\zeta)
 +\frac{w_4^2}{\delta_4^2}(h^6+\zeta)
 \right].
\end{equation}
For any fixed base instance parameters, taking
$\zeta=K^{-4}$ and $K\to\infty$ makes the right-hand side tend to zero:
the cubic term is $O(\log^2K/K)$, the quartic term is
$O(\log^3K/K^2)$, and the exceptional-event contribution is $O(K^{-3})$.
Also $h\to0$, so both core and terminal-certificate thresholds can be met
by a finite dimension.

For the normalized chain regime of Proposition~\ref{ana:prop:calibration}, the following more explicit
choice is sufficient.  Suppose the constants below may depend on the fixed
finite smoothness order $p$, but not on $\varepsilon$:
\begin{align*}
 R&\asymp_p\varepsilon^{-1},&
 A&\asymp_p\varepsilon^{-1/3},&
 r&\asymp_p\varepsilon^{-1/3},\\
 N&\asymp_p\varepsilon^{-4/3},&
 D_0&\asymp_p\varepsilon^{-5/3},&
 T,L&=O_p(\varepsilon^{-5/3}),\\
 |w_3|&=O_p(\varepsilon^{1/3}),&
 |w_4|&=O_p(1),&
 h_{\mathrm{cert}}&\asymp_p\varepsilon.
\end{align*}
Use \eqref{gau:eq:delta-choice} with a fixed positive $\mu$, so
$\delta_3\asymp_p\varepsilon^3$ and
$\delta_4\asymp_p\varepsilon^4$.  Choose
\begin{equation}\label{gau:eq:K-explicit}
 K=\left\lceil C_p\varepsilon^{-14}\right\rceil,
 \qquad \zeta=K^{-4}.
\end{equation}
For sufficiently small $\varepsilon$ the logarithms in
\eqref{gau:eq:geometry-input} are absorbed by the bound
$L+1=O_p(\varepsilon^{-5/3})$, giving
\[
 h^2=O_p(\varepsilon^{-5/3}/K).
\]
Substitution into \eqref{gau:eq:general-calibration} yields
\begin{equation}\label{gau:eq:explicit-kl}
 \mathfrak K\leq C'_p\left[
       \frac{\varepsilon^{-13}}K
       +\frac{\varepsilon^{-52/3}}{K^2}
       +\frac{\varepsilon^{-37/3}}{K^3}
       \right]=O_p(\varepsilon).
\end{equation}
The three exponents follow respectively from
\begin{align*}
 dN^2\frac{w_3^2}{\delta_3^2}h^4
   &=O_p(K\varepsilon^{-29/3}
               \varepsilon^{-10/3}/K^2)
     =O_p(\varepsilon^{-13}/K),\\
 dN^2\frac{w_4^2}{\delta_4^2}h^6
   &=O_p(K\varepsilon^{-37/3}
               \varepsilon^{-5}/K^3)
     =O_p(\varepsilon^{-52/3}/K^2),\\
 dN^2\left(\frac{w_3^2}{\delta_3^2}
          +\frac{w_4^2}{\delta_4^2}\right)\zeta
   &=O_p(\varepsilon^{-37/3}/K^3).
\end{align*}
In particular,
\[
 h=O_p(\varepsilon^{37/6}),\qquad
 Rh=O_p(\varepsilon^{31/6}),\qquad
 d=O_p(\varepsilon^{-47/3}).
\]
Thus $Rh\leq h_{\mathrm{cert}}$ and $ARh\leq r$ hold for all sufficiently
small $\varepsilon$.  Proposition~\ref{gau:prop:transfer} then bounds average
success by $O_p(\sqrt\varepsilon)$, once the deterministic hypotheses are identified with Theorem~\ref{ana:thm:confinement}. The next subsection makes this identification and completes the quantifiers.

\subsection{Assembly of the fixed hard distribution and all quantifiers}
\label{sec:assembly}

We now specify one distribution, chosen independently of any algorithm.
\(Q_0\) is the \(NM\times NM\) pure quadratic matrix of the active
block chain; its linear term is \(c_0=(-\lambda e,0,\ldots,0)\).
There are \(K+1\) copies of \(Q_0\), only the first of which has a
linear term or scalar potentials. The canonical active subspace is
\(\mathcal A_0=\R^{D_0}\oplus0\), where \(D_0=NM\).
For each phase \(s\), let \(\bar b_s\) be its canonical unit readout;
after the Haar rotation it is \(b_s=U\bar b_s\), and \(a_s=A b_s\).
The scalar \(b\) without a subscript denotes the small analytic
derivative budget of Theorem~\ref{ana:thm:confinement}.

\subsubsection{Public parameters and deterministic analytic data}

Fix the constants required by Theorem~\ref{ana:thm:confinement} and
use Proposition~\ref{ana:prop:calibration}. Take \(C_R=1\), so
\[
 R=\varepsilon^{-1},\qquad
 \eta_0=4\varepsilon^2,\qquad
 \delta_3=\varepsilon^3/8,\qquad
 \delta_4=\varepsilon^4/8.
\]
Then
\(\delta_3R^2+\delta_4R^3=\varepsilon/4\), exactly the permitted
mask-gradient budget. Choose
\[
 j_\star=\lceil M/2\rceil,\qquad
 T=\left\lfloor \frac{N j_\star}{2}\right\rfloor,\qquad
 L_{\rm cand}=T+1.
\]
For sufficiently small \(\varepsilon\), \(Nj_\star\ge4\), and therefore
\[
 T+1<Nj_\star,\qquad
 T=\Theta_p(\varepsilon^{-5/3}).
\]
The symbol \(L_{\rm cand}\) counts the final unqueried candidate in the
geometry theorem; the path stiffness \(L=b/16\) in the analytic
calibration is a different, fixed constant.

Set
\[
 K=\lceil \varepsilon^{-14}\rceil,\qquad
 d=(K+1)D_0,\qquad \zeta=K^{-4},\qquad
 \overline Q=I_{K+1}\otimes Q_0 .
\]
Multiplying \(K\) by a fixed \(p\)-dependent constant is also permissible.
The choice above suffices after reducing the positive threshold
\(\varepsilon_p\), and all constants below may depend only on \(p\).
Let \(J\le D_0\) be the number of distinct eigenvalues of \(Q_0\), and put
\[
 \eta^2=
 \frac{32[(L_{\rm cand}+1)\log17+\log(4J/\zeta)]}{K+1},
 \qquad \theta=2\eta .
\]
For sufficiently small \(\varepsilon\), the dimension and multiplicity
conditions of Theorem~\ref{geo:thm:hidden} hold.

\subsubsection{One static distribution}

Draw \(U\) Haar on \(O(d)\). Independently draw standard isonormal
Gaussian tensors \(Z_3\in\Sym^3(\R^d)\) and
\(Z_4\in\Sym^4(\R^d)\), and set
\[
 \sigma_k=\frac{\delta_k}{a_k\sqrt d},\qquad
 a_k=4\sqrt{k\log(1+4k)+2}\quad (k=3,4).
\]
Write \(y=U^\top z\), partitioned into \(K+1\) copies of \(\R^{D_0}\),
and define the rotated deterministic base
\[
 F_U(z)=F(y^{(1)})
       +\frac12\sum_{j=2}^{K+1}
                       (y^{(j)})^\top Q_0 y^{(j)}.
\]
Define the static perturbation and final global function
\begin{align}
 P_Z(z)&=\frac{\sigma_3}{3}Z_3[z,z,z]
          +\frac{\sigma_4}{4}Z_4[z,z,z,z],\\
 f_{U,Z}(x)&=(F_U+P_Z)(\rho_R(x))
                +\frac{\eta_0}{2}\|x\|^2.
 \label{eq:final-instance}
\end{align}
Finally condition the distribution of \((U,Z_3,Z_4)\) on
\[
 \mathcal G_Z=\{\sigma_3\|Z_3\|_{\op}\le\delta_3,\
                 \sigma_4\|Z_4\|_{\op}\le\delta_4\}.
\]
This is \(\mathcal D_{p,\varepsilon}\).
Lemma~\ref{gau:lem:noise-norm} gives
\(\Pr(\mathcal G_Z)\ge1-4e^{-4d}\).
Theorem~\ref{ana:thm:confinement} applies pointwise on this event:
\[
 f_{U,Z}\in C^\infty,\qquad
 f_{U,Z}(0)-\inf f_{U,Z}\le1/3,\qquad
 \sup_x\|D^{q+1}f_{U,Z}(x)\|_{\op}\le1
 \quad(1\le q\le p).
\]
Thus the distribution is supported entirely on
\(\mathcal H_{p,d}\). Sampling and conditioning occur before the
algorithm's private seed is used.

The same theorem supplies a single terminal certificate, valid on
\(\mathcal G_Z\) for every \(x\in\R^d\):
\begin{equation}\label{eq:terminal-certificate}
 |b_N^\top\rho_R(x)|\le h_{\rm cert}:=16\varepsilon/b
 \quad\Longrightarrow\quad
 \|\nabla f_{U,Z}(x)\|>\varepsilon .
\end{equation}
For \(\|x\|\le R\) this is the approximate-tail conclusion; for
\(\|x\|\ge R\) the radial certificate is even stronger and does not
require the readout condition.

\subsubsection{The exact strengthened experiment}

The two comparison laws \(\Ptrue\) and \(\Pref\) in this subsection use
the unconditioned independent Gaussian prior. The conditioned admissible
distribution \(\mathcal D_{p,\varepsilon}\) enters only in the final
transfer through \(\mathcal G_Z\).

Fix any algorithm with at most \(T\) calls on all answer histories.
At time \(t\), it chooses \(x_t\), and set \(z_t=\rho_R(x_t)\).
The quadratic channel supplies \(Q_Uz_t\), where
\(Q_U=U\overline Q U^\top\). Give the known linear vector \(Uc_0\)
for free, with \(c_0\) embedded in the first copy.
The tensor channels of Theorem~\ref{gau:thm:innovation} have
\[
 S_3=-120\tau A^3\sum_{s=1}^N b_s^{\otimes3},\qquad
 S_4=120\tau A^4\sum_{s=1}^N b_s^{\otimes4},
 \qquad \Theta_k=S_k+\sigma_kZ_k .
\]
The exact degree-\(k\) channel is
\(\Theta_k[z_t^{\otimes(k-1)},\cdot]\).
Their scalar homogeneous values are recovered by contraction with \(z_t/k\);
there is no extra value information to account for.

For the reference's side information, assign active coordinate
\((s,j)\) layer
\[
 \ell(s,j)=(s-1)j_\star+\min\{j,j_\star\},
\]
and let \(I_t\) be the canonical span of all active coordinates
with layer at most \(t\). At the start of time \(t\), disclose \(U|_{I_t}\).
The initial linear vector lies in \(I_1\), and direct inspection gives
\(\overline Q I_t\subseteq I_{t+1}\).
The full suffix readout \(\bar b_s\) is orthogonal to \(I_t\) for
\(t<sj_\star\) and lies in \(I_t\) for \(t\ge sj_\star\).
Use \(\tau_s=sj_\star\) if \(sj_\star\le T+1\), and otherwise use
\(\tau_s=\infty\). This agrees with the finite-horizon convention of
Theorem~\ref{gau:thm:innovation}.

As in \eqref{gau:eq:assembled-value}--\eqref{gau:eq:assembled-grad},
assemble a value-gradient reply from the quadratic channel, the two
tensor channels, and only the already disclosed taper corrections.
This defines the true polynomial-channel comparison law \(\Ptrue\).
The tested policy is adapted only to these assembled value-gradient pairs.
The separated tensor channels, scheduled disclosures, latent-coordinate
responses, and quadratic matrix-vector products are analyst-side variables
in an enlarged transcript; formally, extend the policy to ignore all of
them. Thus the strengthened experiment does not give the tested algorithm a
stronger oracle. Before any hidden-core exit the assembled reply agrees
exactly with the oracle of \eqref{eq:final-instance}.

Under the innovation-censored law \(\Pref\), fix all independent physical
reference Gaussian seeds \(\zeta_{k,t}\) and the algorithm seed.
The fitter, observation projectors, covariance, and visible tensor means
are functions of the physical history and disclosed \(U|_{I_t}\) only.
The reference therefore satisfies the exact policy model of
Theorem~\ref{geo:thm:hidden}. No full static mask tensor is conditioned
on in this step.
The radial map changes a query to a known scalar multiple, so
one quadratic matrix-vector product per original query suffices.

Apply the geometric model with
\[
 L=L_{\rm cand}=T+1,\qquad L_{\rm q}=T.
\]
Include \(z_{\rm out}=\rho_R(x_{\rm out})\) as the response-free
candidate at time \(T+1\). It is frozen from the first \(T\) assembled
oracle replies and receives no subsequent response. Only after it is frozen
does the analyst invoke the scheduled disclosure at \(I_{T+1}\), which the
terminated policy ignores. The response-free branch of the geometric theorem
therefore gives, with reference
probability at least \(1-\zeta\), simultaneously
\[
 \|P_{W_s}b_s\|\le\theta\quad\text{for all }s,\qquad
 |b_N^\top z_{\rm out}|\le R\theta .
\]
There is no union bound over adaptively selected query points.
The whole realized query span is controlled, including arbitrary
near-dependencies and exact rescalings.

\subsubsection{Information budget, core agreement, and failure probability}

The calibration gives
\[
 |w_3|=O_p(\varepsilon^{1/3}),\qquad
 |w_4|=O_p(1),\qquad
 N=\Theta_p(\varepsilon^{-4/3}),\quad
 D_0,L_{\rm cand}=\Theta_p(\varepsilon^{-5/3}).
\]
Because \(K=\Theta(\varepsilon^{-14})\), the geometric threshold satisfies
\[
 \theta^2=O_p(\varepsilon^{-5/3}/K),\qquad
 R\theta=O_p(\varepsilon^{31/6}).
\]
In particular, for sufficiently small \(\varepsilon\),
\[
 AR\theta\le r=A,\qquad
 R\theta\le16\varepsilon/b=h_{\rm cert}.
\]
Thus all undisclosed readouts are inside their exact polynomial cores,
and the final hidden readout stays inside the stationarity obstruction,
on the reference geometric event.

Equations~\eqref{gau:eq:kl-angle} and \eqref{gau:eq:explicit-kl} give
\begin{equation}\label{eq:assembled-KL}
 \KL(\Pref\Vert\Ptrue)
 \le C_p\left[
       \frac{\varepsilon^{-13}}{K}
       +\frac{\varepsilon^{-52/3}}{K^2}
       +\frac{\varepsilon^{-37/3}}{K^3}\right]
 =O_p(\varepsilon).
\end{equation}
This is a bound for exact static tensor observations. It remains valid
if queries are repeated, nearly dependent, or chosen at arbitrarily
small finite-difference scales.

The first-exit coupling and Proposition~\ref{gau:prop:transfer} imply
\[
 \Pr_{\mathcal D_{p,\varepsilon},\,\omega}
 \{\|\nabla f_{U,Z}(x_{\rm out})\|\le\varepsilon\}
 \le
 \frac{\zeta+
       \sqrt{\KL(\Pref\Vert\Ptrue)/2}}
      {1-4e^{-4d}}
 =O_p(\sqrt\varepsilon).
\]
Reduce \(\varepsilon_p>0\) so this is at most \(1/3\), in addition to
the finitely many earlier small-accuracy requirements.
All such thresholds and constants depend only on \(p\).
The dimension is
\[
 d=(K+1)NM=O_p(\varepsilon^{-47/3}).
\]
This proves the fixed-distribution part of Theorem~\ref{thm:matching}.

\subsubsection{Almost-sure caps, private randomness, and a fixed bad instance}

The reference argument is uniform over all measurable policies with a
pathwise query cap \(T\). Suppose instead that the original algorithm is
specified with an almost-sure cap on each admissible fixed function,
as in the question. Truncate it after \(T\) calls on every answer
history, supplying a fixed dummy output if it would attempt another call.
On every admissible fixed function this changes its law on a
probability-zero set of seeds only. Integrating over
\(\mathcal D_{p,\varepsilon}\) preserves that fact. Thus the lower
bound also applies to the original almost-sure convention.

The distribution itself depends on \(p,\varepsilon\), the explicit
template, and its own sampled \(U,Z_3,Z_4\); it does not depend on the
algorithm's seed. We may include that independent seed in the physical
advice of the geometric proof and average afterwards. Equivalently,
the same average failure bound holds for every fixed-seed policy and
then for their mixture.
If an algorithm succeeded with probability at least \(2/3\) on every
fixed \(f\in\mathcal H_{p,d}\), its average success under
\(\mathcal D_{p,\varepsilon}\) would be at least \(2/3\), contradicting
the preceding bound of \(1/3\). Hence one fixed admissible function
has success below \(2/3\).

Finally \(T=\Theta_p(\varepsilon^{-5/3})\); choosing \(c_p>0\)
smaller than its asymptotic lower coefficient gives the claimed
query lower bound. No expected-cost assumption, adaptive hard instance,
oracle noise resampling, or queried-output restriction has entered
the argument.

\section{Compatibility with the matching upper bound}\label{app:upper}
\label{sec:upper}

Section~1.4 of Carmon et al.~\cite{carmon} defines \(L_q\)-Lipschitz
\(q\)th derivative by exactly the directional line condition used here.
Thus \(f\in\mathcal H_{p,d}\) has their smoothness bounds
\(L_1=L_3=1\), with no tensor-norm conversion.
Their Theorem~2 gives a deterministic exact first-order method with
\[
 O\!\left(
 \Delta_f L_1^{1/2}L_3^{1/6}\varepsilon^{-5/3}
 \log\!\frac{500L_1\Delta_f}{\varepsilon^2}
 \right)
\]
gradient evaluations, where \(\Delta_f\) is the actual initial gap. Their
evaluation accounting gives the same asymptotic order of function
evaluations. Hence, under the joint exact value-gradient oracle used here,
the number of oracle calls has the displayed order, provided
\[
 0<\varepsilon^{2/3}
 \le \min\{\Delta_f^{1/2}L_3^{1/6},
                  L_1/(8L_3^{1/3})\}.
\]
Its algorithmic parameters do not require knowledge of \(\inf f\).

To make this condition uniform over functions whose actual gap may be
arbitrarily small, simulate the algorithm on the explicit virtual lift
\[
 \widetilde f(x,t)=f(x)+\frac12(t-1)^2,
 \qquad (x_0,t_0)=(0,0).
\]
The actual lifted gap belongs to \([1/2,3/2]\).
Its gradient is \(1\)-Lipschitz, and its directional third derivative
is \(1\)-Lipschitz: for higher derivatives the additional quadratic
vanishes, and the original-direction component has norm at most one.
The theorem's accuracy condition now holds uniformly when
\(\varepsilon^{2/3}\le1/8\).
One lifted value-gradient query requires exactly one original oracle
query, since the extra quadratic is known explicitly.
Moreover,
\[
 \|\nabla\widetilde f(x,t)\|^2
   =\|\nabla f(x)\|^2+(t-1)^2,
\]
so an \(\varepsilon\)-stationary lifted point gives an
\(\varepsilon\)-stationary original point by dropping its last coordinate.
The actual gap bound \(3/2\) therefore gives
\[
 \Qdet(\varepsilon)
 \le C\varepsilon^{-5/3}\log(2/\varepsilon).
\]
Its deterministic success also proves the randomized upper bound.
Together with Section~\ref{sec:assembly}, this completes the proof
of Theorem~\ref{thm:matching}.

\end{document}